\documentclass{amsart}

\newcommand{\Hom}{\underline{\mathrm{Hom}}}

\usepackage[toc,page]{appendix}
\usepackage[headheight=12.1pt]{geometry}
\usepackage[all,cmtip]{xy}
\usepackage{amsmath}
\usepackage{amsfonts}
\usepackage{amssymb}
\usepackage{mathrsfs}   
\usepackage{amsthm}
\usepackage{xcolor}
\usepackage{cite}
\usepackage{stmaryrd}
\usepackage{graphicx}
\usepackage{pgf}
\usepackage{eepic}
\usepackage{pstricks}
\usepackage{epsfig}
\usepackage{fancyhdr}
\usepackage{tikz,caption,float}
\usepackage[backref=page]{hyperref}

\usepackage{amsbsy}

\numberwithin{equation}{subsection}

\let\realequation\equation
\def\equation{\setcounter{equation}{\arabic{subsubsection}}%
   \refstepcounter{subsubsection}%
   \realequation}

\usepackage{xr,mathtools}

\hypersetup{
     colorlinks   = true,
     citecolor    = red,
     linkcolor = blue,
     urlcolor = purple
}
    
  \newcommand\imCMsym[4][\mathord]{%
  \DeclareFontFamily{U} {#2}{}
  \DeclareFontShape{U}{#2}{m}{n}{
    <-6> #25
    <6-7> #26
    <7-8> #27
    <8-9> #28
    <9-10> #29
    <10-12> #210
    <12-> #212}{}
  \DeclareSymbolFont{CM#2} {U} {#2}{m}{n}
  \DeclareMathSymbol{#4}{#1}{CM#2}{#3}
}
\newcommand\alsoimCMsym[4][\mathord]{\DeclareMathSymbol{#4}{#1}{CM#2}{#3}}

\imCMsym{cmmi}{124}{\CMjmath}
\imCMsym[\mathop]{cmsy}{113}{\CMamalg}
\imCMsym[\mathop]{cmex}{96}{\CMcoprod}
\alsoimCMsym[\mathop]{cmex}{97}{\CMbigcoprod}
 
\swapnumbers
\theoremstyle{plain}
\newtheorem*{theoremu}{Theorem}
\newtheorem{theorem}[subsubsection]{Theorem}

\newtheorem{proposition}[subsubsection]{Proposition}

\newtheorem{corollary}[subsubsection]{Corollary}

\newtheorem{lemma}[subsubsection]{Lemma}

\theoremstyle{definition}

\newtheorem{definition}[subsubsection]{Definition}

\theoremstyle{remark}
\newtheorem{remark}[subsubsection]{Remark}

\newtheorem{example}[subsubsection]{Example}

\newcommand{\N}{{\mathbb N}}
\newcommand{\Z}{{\mathbb Z}}
\newcommand{\Q}{{\mathbb Q}}
\newcommand{\R}{{\mathbb R}}

\newcommand{\A}{{\mathbb A}}
\newcommand{\D}{{\mathbb D}}

\newcommand{\scr}[1]{\mathscr{#1}}

\newcommand{\cO}{\mathcal{O}}
\newcommand{\cA}{\mathcal{A}}

\newcommand{\fr}[1]{\mathfrak{#1}}

\renewcommand{\lim}[1]{\underset{#1}{\mathrm{lim}}\,}

\newcommand{\isomto}{\overset{\cong}{\longrightarrow}}

\newcommand{\from}{\colon}
\renewcommand{\to}{\rightarrow}

\newcommand{\lto}{\longrightarrow}

\newcommand{\an}{\mathrm{an}}

\newcommand{\pow}[1]{\llbracket #1 \rrbracket}

\newcommand{\weak}[1]{\langle #1\rangle^\dagger}
\newcommand{\tate}[1]{\langle #1 \rangle}

\newcommand{\spec}[1]{\mathrm{Spec}\left(#1\right)}
\newcommand{\spa}[1]{\mathrm{Spa}\left(#1\right)}
\newcommand{\spf}[1]{\mathrm{Spf}\left(#1\right)}

\newcommand{\norm}[1]{\left\vert#1\right\vert}

\newcommand{\Tr}{\mathrm{Tr}}

\title{Proper pushforwards on analytic adic spaces}

\SetSymbolFont{stmry}{bold}{U}{stmry}{m}{n}
\usepackage{bm}
\usepackage{scalerel}

\author{Tomoyuki Abe}
       \address{Kavli Institute for the Physics and Mathematics of the Universe (WPI) \\ University of Tokyo
 \\ 5-1-5 Kashiwanoha \\ Kashiwa \\ Chiba \\ 277-8583 \\ Japan}
       \email{tomoyuki.abe@ipmu.jp}

\author{Christopher Lazda}
       \address{Department of Mathematics\\ Harrison Building \\ Streatham Campus
 \\ University of Exeter \\ North Park Road \\ Exeter  \\ EX4 4QF \\ United Kingdom }
       \email{c.d.lazda@exeter.ac.uk}

\begin{document}

\begin{abstract} We construct proper pushforwards for partially proper morphisms of analytic adic spaces. This generalises the theory due to van\thinspace der\thinspace Put in the case of rigid analytic varieties over a non-Archimedean field. For morphisms which are smooth and partially proper in the sense of Kiehl, we furthermore construct the trace map and duality pairing.
\end{abstract}

\maketitle 

\tableofcontents

\section*{Introduction}

Arguably the most significant early achievement of Huber's approach to analytic geometry, via his theory of adic spaces, was that it enabled the development of a robust theory of \'etale cohomology with compact support for rigid analytic varieties, and in particular the proof in \cite{Hub96} of a Poincar\'e duality theorem for smooth morphisms. Previously, van\thinspace der\thinspace Put in \cite{vdP92} constructed a theory of compactly supported cohomology and proper pushforwards for abelian sheaves on rigid analytic varieties over a non-Archimedean field, and proved a version of Serre duality for smooth and proper morphisms. 

Our first goal in this article is to recast van\thinspace der\thinspace Put's definition in Huber's language of adic spaces, and therefore generalise it to the case of analytic adic spaces which are not necessarily defined over a field. In fact, the definition is very simple: if $f\colon X\rightarrow Y$ is a partially proper morphism of analytic adic spaces, we simply define $\mathbf{R}f_!$ to be the derived functor of sections whose support is quasi-compact (and hence proper) over $Y$. The important thing is to show that these compose correctly, and the proof that they do so follows the strategy of \cite[Chapter 5]{Hub96} very closely. 

The main application we have in mind for the formalism developed here is to the theory of rigid cohomology, which necessitates working not just with adic spaces, but with \emph{germs} of adic spaces, i.e. closed subsets of adic spaces with the `induced' analytic structure. (For us, the motivating example of a germ is the tube $]X[_\fr{P}$ of a locally closed subset $X$ of a formal scheme $\mathfrak{P}$.) This generalisation is similar in spirit to the `pseudo-adic spaces' that Huber works with in \cite{Hub96}, although far more modest in scope. 

Our second goal is to define the trace map for smooth morphisms which are `partially proper in the sense of Kiehl' (see Definition \ref{defn: pp Kiehl}).

\begin{theoremu}[\ref{prop: tr basic}, \ref{prop: tr1}, \ref{prop: trace main}] There exists a unique way to define, for any smooth morphism $f:X\rightarrow Y$ of relative dimension $d$, partially proper in the sense of Kiehl, with $Y$ an overconvergent\footnote{A germ is overconvergent if it is stable under generalisation inside its ambient adic space.} and finite dimensional analytic germ, an $\mathcal{O}_Y$-linear map
\[  \mathrm{Tr}_{X/Y}: \mathbf{R}^df_!\omega_{X/Y} \rightarrow \mathcal{O}_Y, \]
such that:
\begin{enumerate}
\item $\Tr_{X/Y}$ is local on the base $Y$, and compatible with composition;
\item when $f$ is an open immersion, $\Tr_{X/Y}$ is the canonical map
\[ f_!\mathcal{O}_X \rightarrow \mathcal{O}_Y; \]
\item when $t\in \Gamma(X,\cO_X)$ is such that $Z:=V(t)$ is also smooth over $Y$, $u\from Z\to X$ denotes the inclusion, and $\alpha\from u_*\omega_{Z/Y}\to \omega_{X/Y}[1]$ the map classifying the exact sequence
\[ 0 \to \omega_{X/Y}\overset{\times t}{\lto}\omega_{X/Y} \to u_*\omega_{Z/Y} \to 0, \]
the diagram
\[ \xymatrix{  \mathbf{R}^{d-1}f_!\omega_{Z/Y}\ar[rr]^{{\bf R}^{d-1}f_!(\alpha)}\ar[drr]_{\Tr_{Z/Y}} & & \mathbf{R}^df_!\omega_{X/Y} \ar[d]^{\Tr_{X/Y}} \\ & &\cO_Y } \]
commutes;
\item when $Y=\spa{R,R^+}$ is affinoid, and $X=\D_Y(0;1^-)$ is the relative open unit disc, with co-ordinate $z$, then, via the identification
\[ {\rm H}^1_c(X/Y,\Omega_{X/Y}^1) \isomto R\weak{z^{-1}}\,{\rm d}\!\log z, \]
$\Tr_{X/Y}$ is given by
\[ \sum_{i\leq 0} r_iz^i \,{\rm d}\!\log z \mapsto r_0.\]
\end{enumerate}
Moreover, $\Tr_{X/Y}$ descends uniquely to a map
\[ \mathbf{R}f_!\Omega_{X/Y}[2d] \to \cO_Y \]
in the derived category of $\cO_Y$-modules. 
\end{theoremu}

Broadly speaking, the construction of $\mathrm{Tr}_{X/Y}$ follows the same outline as in \cite{vdP92}. First we work with relative open unit polydiscs, then closed subspaces of relative open unit polydiscs, and then finally show that the map $\mathrm{Tr}_{X/Y}$ thus constructed doesn't depend on the choice of embedding, and therefore globalises.

The question of what form of Serre--Grothendieck duality holds, and in what generality, we do not address here at all. Our main motivation for developing a formalism of proper pushforwards, and for constructing trace morphisms, was to understand particular constructions in rigid cohomology and the theory of arithmetic $\mathscr{D}$-modules \cite{AL2}. For us, it was enough simply to have the formalism (together with an explicit computation for relative open unit polydiscs); a detailed study of duality would therefore have distracted us rather too much from our main goal. Another natural question to ask is whether or not the formalism of $\mathbf{R}f_!$ extends in any reasonable way beyond the partially proper case. We give an example in \S\ref{sec: counter examples} to show that this cannot be done, essentially for the same reason as in the case of abelian sheaves on the Zariksi site of schemes, namely, the failure of the proper base change theorem. We thank B. Le Stum for the main idea behind this example.

Let us now give a brief summary of the contents of this article. In \S\ref{sec: gt} we gather together various (existing) results in general topology that will be useful in the rest of the article, in particular concerning sheaf cohomology on spectral spaces. In \S\ref{sec: qad} we introduce the notion of a germ of an adic space along a closed subset, and define the category in which these live. In \S\ref{sec: defn} we define proper pushforwards $f_!$ and $\mathbf{R}f_!$ for partially proper morphisms of germs of adic spaces, and show that these derived proper pushforwards compose correctly. In \S\ref{sec: kpp} we prove a result on the cohomological dimension of coherent sheaves, which is then used in \S\ref{sec: trace} to constructed the trace map for smooth morphisms which are partially proper in the sense of Kiehl. Finally, in \S\ref{sec: counter examples} we give an example to show that there is no good formalism of $\mathbf{R}f_!$ beyond the partially proper case.

\subsection*{Acknowledgements}

We would like to thank Yoichi Mieda for reading the manuscript carefully, and giving us several helpful comments, and J\'er\^ome Poineau for answering some of our questions. We would like to thank Bernard Le Stum for pointing out the failure of the proper base change theorem in general. We would also like to thank the anonymous referee for a thorough reading of the manuscript, and for their helpful suggestions to improve it. The first author was supported by JSPS KAKENHI Grant Numbers 16H05993, 18H03667, 20H01790. 

\subsection*{Notation and conventions}

We will only deal with abelian sheaves. Thus if $X$ is a topological space, a sheaf on $X$ will always mean an abelian sheaf. The category of abelian sheaves on $X$ will be denoted by $\mathbf{Sh}(X)$, and its derived category by ${\bf D}(X)$. If $\mathcal{A}$ is a sheaf of rings on $X$, the derived category of $\cA$-modules will be denoted ${\bf D}(\cA)$.

A Huber ring will be a topological ring admitting an open adic subring $R_0$ with finitely generated ideal of definition. For such a ring $R$, we will denote by $R^\circ \subset R$ the subset of power-bounded elements, and by $R^{\circ\circ}\subset R^\circ$ the subset of topologically nilpotent elements. A Huber pair is a pair $(R,R^+)$ consisting of a Huber ring $R$ and an open, integrally closed subring $R^+\subset R^\circ$. A Huber ring $R$ is said to be a Tate ring if there exists some $\varpi\in R^\times\cap R^{\circ\circ}$, such an element will be called a quasi-uniformiser. A Huber pair $(R,R^+)$ is said to be a Tate pair if $R$ is a Tate ring. An adic space isomorphic to $\spa{R,R^+}$ with $(R,R^+)$ a Tate pair will be called a Tate affinoid.

If $X$ is an adic space, and $x\in X$, we will denote by $\mathcal{O}_{X,x}$ and $\mathcal{O}_{X,x}^+$ the stalks of the structure sheaf and integral structure sheaf of $X$ at $x$ respectively. The residue field of $\mathcal{O}_{X,x}$ will be denoted $k(x)$, and the image of $\mathcal{O}_{X,x}^+$ in $k(x)$ by $k(x)^+$. The residue field of $k(x)^+$ will be denoted $\widetilde{k(x)}$.

\section{General topology} \label{sec: gt}

In this section we will gather together various existing definitions and results that we will need from general topology, mostly using either \cite{FK18} or \cite{Hub96} as references.

\subsection{Basic definitions} For the reader's convenience we recall several definitions that will be used in this article.

\begin{definition} A topological space $X$ is said to be:
\begin{enumerate}
\item quasi-compact if every open cover has a finite sub-cover;
\item compact if it is quasi-compact and Hausdorff;
\item locally compact if every point has a compact neighbourhood;
\item quasi-separated if the intersection of any two quasi-compact opens is quasi-compact;
\item coherent if it is quasi-compact, quasi-separated, and admits a basis of quasi-compact open subsets;
\item locally coherent if it admits a cover by coherent open subspaces
\item sober if every irreducible closed subset has a unique generic point; 
\item spectral if it is coherent and sober;
\item locally spectral if it is locally coherent, and sober.
\item valuative if it is locally spectral, and the set of generalisations of any point $x\in X$ is totally ordered;
\item taut if it is locally spectral, quasi-separated, and the closure of any quasi-compact open $U\subset X$ is quasi-compact.
\end{enumerate}
\end{definition} 

If $X$ is a locally spectral space, and $x,y\in X$ we will write $y\succ x$ if $x\in \overline{\{y\}}$, i.e. if $x$ is a specialisation of $y$. We will also write $G(x)$ for the set of generalisations of $x$.

\begin{definition} A morphism $f:X\rightarrow Y$ of locally coherent topological spaces is said to be
\begin{enumerate}
\item quasi-compact if the preimage of every quasi-compact open subset $V\subset Y$ is quasi-compact;
\item quasi-separated if the preimage of every quasi-separated open subset $V\subset Y$ is quasi-separated;
\item coherent if it is quasi-compact and quasi-separated;
\end{enumerate}
A morphism $f:X\rightarrow Y$ of locally spectral spaces is said to be
\begin{enumerate}
\item[(4)] taut if the pre-image of every taut open subspace $V\subset Y$ is taut;
\item[(5)] spectral if for every quasi-compact and quasi-separated open subsets $U\subset X$, $V\subset Y$ with $f(U)\subset V$, the induced map $f:U\rightarrow V$ is quasi-compact.
\end{enumerate} 
\end{definition}

\begin{definition} A morphism $f:X\rightarrow Y$ of topological spaces is said to be topologically proper if for all topological spaces $Z$, the map $X\times Z \rightarrow Y\times Z$ is closed. 
\end{definition}

\begin{remark} We use the terminology topologically proper to distinguish this from the analytic notion of properness that we will use later on. 
\end{remark}

If $f$ is topologically proper, then preimages of quasi-compact sets are quasi-compact \cite[\S10.2, Proposition 6]{Bou98}, and if $X$ is Hausdorff and $Y$ locally compact, then the converse holds \cite[\S10.3, Proposition 7]{Bou98}. 

\subsection{Sheaf cohomology on spectral spaces}

The following result will be used constantly.

\begin{proposition} \label{prop: cohom limits} Let $X$ be a topological space, $\{U_i\}_{i\in I}$ a filtered diagram of open subsets of $X$, such that each $U_i$ is spectral, and set $Z=\cap_{i\in I}U_i$. Then, for any sheaf $\mathscr{F}$ on $X$, and any $q\geq0$, the natural map
\[ \mathrm{colim}_{i\in I} {\rm H}^q(U_i,\mathscr{F}|_{U_i}) \rightarrow {\rm H}^q(Z,\mathscr{F}|_{Z}) \]
is an isomorphism.
\end{proposition}

\begin{proof}
Since the inclusions $U_i\rightarrow U_j$ are automatically quasi-compact by \cite[Chapter 0, Proposition 2.2.3]{FK18}, this is a particular case of \cite[Chapter 0, Proposition 3.1.19]{FK18}. 
\end{proof}

\begin{corollary} \label{cor: pushforward stalks} Let $f:X\rightarrow Y$ be a coherent morphism of locally spectral spaces, $\mathscr{F}$ a sheaf on $X$, and $y\in Y$. Let $X_{(y)}\subset X$ denote the inverse image of $G(y)\subset Y$. Then, for any $q\geq0$, the natural map
\[ (\mathbf{R}^qf_*\mathscr{F})_y \rightarrow {\rm H}^q(X_{(y)},\mathscr{F}|_{X_{(y)}})  \]
is an isomorphism.
\end{corollary}

\begin{proof}
We may assume that $Y$ is coherent, thus spectral. Hence $X$ is also spectral. The point $y$ admits a cofinal system of open neighbourhoods $\{ U_i\}_{i\in I}$ with each $U_i$ spectral. Therefore each $f^{-1}(U_i)$ is spectral, and we have
\[(\mathbf{R}^qf_*\mathscr{F})_y = \mathrm{colim}_{i\in I} {\rm H}^q(f^{-1}(U_i),\mathscr{F}|_{f^{-1}(U_i)}) = {\rm H}^q(X_{(y)},\mathscr{F}|_{X_{(y)}}) \] 
since $\cap_i f^{-1}(U_i) = f^{-1}(\cap_i U_i) = f^{-1}(G(y))= X_{(y)}$. 
\end{proof}

\subsection{Dimensions of spectral spaces} The dimension theory of locally spectral spaces works as in the case of schemes.

\begin{definition} Let $X$ be a locally spectral space. The dimension of $X$ is defined to be
\[ \dim X = \sup\{\left. n \geq 0 \;\right\vert\; \exists\;\; x_n \succ x_{n-1} \succ \ldots \succ x_0,\;\;x_i\neq x_{i-1}   \} \in \Z_{\geq0} \cup \{\infty,-\infty\}.  \]
The space $X$ is said to be finite dimensional if $\dim X <\infty$. 
\end{definition}

We will need the following generalisation of Grothendieck vanishing.

\begin{theorem} \label{theo: cd ss} Let $X$ be a spectral space of dimension $d$, and $\mathscr{F}$ a sheaf on $X$. Then ${\rm H}^q(X,\mathscr{F})=0$ for all $q>d$.
\end{theorem}

\begin{proof}
This is the main result of \cite{Sch92}.
\end{proof}

\begin{definition} Let $f:X\rightarrow Y$ be a spectral morphism between locally spectral spaces. The dimension of $f$ is defined to be
\[ \dim f = \sup\{\left. \dim f^{-1}(y)\right\vert y\in Y  \} \in \Z_{\geq0} \cup \{\infty,-\infty\}.  \]
The map $f$ is said to be finite dimensional if $\dim f <\infty$. 
\end{definition}

\subsection{Separated quotients}

Let $X$ be a valuative space, and let $[X]$ denote its set of maximal points, i.e. points such that $G(x)=\{x\}$. Then taking a point to its maximal generalisation\footnote{That every point of a valuative space does indeed admit a maximal generalisation is explained in \cite[Chapter 0, Remark 2.3.2]{FK18}.} induces a surjective map
\[ \nu : X\rightarrow [X], \]
and we equip $[X]$ with the quotient topology.

\begin{proposition} The space $[X]$ is $T_1,$\footnote{A space is $T_1$ if for any two distinct points, each has an open neighbourhood not containing the other.} and is universal for maps from $X$ into $T_1$ topological spaces. If $X$ is coherent, then $[X]$ is compact.
\end{proposition}

\begin{proof}
This is \cite[Chapter 0, Proposition 2.3.9 and Corollary 2.3.18]{FK18}.
\end{proof}

Note that the space $[X]$ is generally no longer valuative, since it won't admit a basis of quasi-compact opens. 

\begin{definition} An open (resp. closed) subset of a valuative space $X$ is said to be \emph{overconvergent} if it is closed under specialisation (resp. generalisation). 
\end{definition}

Equivalently, it is the preimage of an open (resp. closed) subset of $[X]$ under the separation map $\nu$. Note that the complement of an overconvergent open subset is an overconvergent closed subset, and vice versa. 

\begin{lemma} \label{lemma: oc closed} Let $Z\subset X$ be an overconvergent closed subset of a coherent valuative space, and $\mathscr{F}$ a sheaf on $X$. Then, for all $q\geq0$, the natural map
\[ \mathrm{colim}_{\substack{Z\subset U\\\mathrm{open}}} {\rm H}^q(U,\mathscr{F}|_U)\rightarrow {\rm H}^q(Z,\mathscr{F}|_Z) \]
is an isomorphism.
\end{lemma}

\begin{proof}
Since $Z$ is the intersection of its open neighbourhoods, this follows from Proposition \ref{prop: cohom limits}.
\end{proof}

\section{Germs of adic spaces}\label{sec: qad}

In this section we introduce the category of germs of adic spaces. This will be the category in which we work for the rest of this article.

\subsection{Standing hypotheses} We use Huber's theory of adic spaces, see either \cite{Hub94} or \cite[Chapter 1]{Hub96} for an introduction. We will assume that all adic spaces are analytic in the sense of \cite[\S1.1]{Hub96}. That is, each point $x\in X$ will have an open affinoid neighbourhood $x\in \spa{R,R^+}\subset X$ such that $R$ is Tate. This implies that all morphisms of adic spaces are adic in the sense of \cite[\S1.2]{Hub96}. We will let $\mathbf{Ad}$ denote the category of analytic adic spaces. Note in particular the standing assumption \cite[(1.1.1)]{Hub96}, which implies that for all complete Huber pairs $(R,R^+)$ we consider, $R$ will be Noetherian.

\subsection{Germs of adic spaces}

In \cite[\S1.10]{Hub96} Huber introduces the notion of a pseudo-adic space, which roughly speaking consists of an adic space $\bm{X}$, together with a `reasonably nice' subspace $X\subset \bm{X}$. We will work instead with germs of adic spaces along closed subsets.

\begin{definition} \label{defn: germ} A germ of an adic space is a pair $(X,\bm{X})$ where $\bm{X}$ is an adic space, and $X\subset \bm{X}$ is a closed subset.
\end{definition}

We can construct a category $\mathbf{Germ}$ of germs of adic spaces in the usual way. We first consider the category of pairs $(X,\bm{X})$ as in Definition \ref{defn: germ}, where morphisms are commutative squares. We then declare a morphism $j:(X,\bm{X})\rightarrow (Y,\bm{Y})$ to be a strict neighbourhood if $j$ is an open immersion and $j(X)=Y$. Finally, we localise the category of pairs at the class of strict neighbourhoods (it is easy to verify that a calculus of right fractions exists). If $(X,\bm{X})$ is a pair, we will often abuse notation and write $X$ for $(X,\bm{X})$ considered as an object in the category $\mathbf{Germ}$.

\begin{example} 
\label{exa: germ}
 \begin{enumerate} 
\item The first key example of a germ is any fibre of a morphism of adic spaces $f:X\rightarrow Y$ which is locally of weakly finite type. Indeed, if $y\in Y$, and $G(y)$ is its set of generalisations, then $f^{-1}(y)$ is a closed subset of the adic space
\[ f^{-1}(G(y))=X\times_Y \spa{k(y),k(y)^+}.\]
Note that if the point $y\in Y$ in question is not maximal, this fibre $f^{-1}(y)$ will not have any kind of `natural' structure as an adic space.
\item \label{num: exa germ 2} The second key example for us (in particular in the forthcoming \cite{AL2}) is inspired by Berthelot's theory of rigid cohomology. Let $k^\circ$ be a complete discrete valuation ring, with fraction field $k$, $\fr{P}$ a formal scheme, flat and of finite type over $\spf{k^\circ}$, and $X\subset \fr{P}$ a locally closed subset. Then there is a (continuous) specialisation map
\[ \mathrm{sp}: \fr{P}_k \rightarrow \fr{P} \]
from the adic generic fibre of $\fr{P}$ to the formal scheme $\fr{P}$, and if we let $Y$ denote the closure of $X$ in $\fr{P}$, then the tube
\[ ]Y[_\fr{P}:= \mathrm{sp}^{-1}(Y)^\circ \]
is defined to be the interior of the inverse image of $Y$ under $\mathrm{sp}$. This induces a map
\[ \mathrm{sp}_Y: ]Y[_\fr{P} \rightarrow Y,\]
and the tube
\[ ]X[_\fr{P}:= \overline{\mathrm{sp}^{-1}_{Y}(X)} \]
is defined to be the closure of $\mathrm{sp}^{-1}_{Y}(X)$ inside $]Y[_{\fr{P}}$. The pair of tubes $(]X[_\fr{P},]Y[_\fr{P})$ then defines a germ, denoted $]X[_\fr{P}$.
\end{enumerate}
\end{example}

The assignment $X\mapsto (X,X)$ induces a fully faithful functor $\mathbf{Ad}\rightarrow \mathbf{Germ}$. We can also consider any pair $(X,\bm{X})$ as in Definition \ref{defn: germ} as a pseudo-adic space in the sense of Huber \cite[\S1.10]{Hub96}, thus it makes sense to consider any of the following properties of morphisms of such pairs:
\begin{enumerate}
\item locally of finite type, locally of $^+$weakly finite type, locally of weakly finite type;
\item quasi-compact, quasi-separated, coherent, taut;
\item an open immersion, closed immersion, locally closed immersion;
\item separated, partially proper, proper;
\item smooth, \'etale.
\end{enumerate}

For example, a morphism of pairs $f:(X,\bm{X}) \rightarrow (Y,\bm{Y})$ is smooth if and only if $f:\bm{X}\rightarrow \bm{Y}$ is smooth in the sense of \cite[\S1.6]{Hub96}, and $X$ is open in $f^{-1}(Y)$. It is easily checked that all of these properties descend to the category $\mathbf{Germ}$ of germs. While `analytic' properties of a germ $X$ are generally defined via the ambient adic space $\bm{X}$, `topological' properties are generally defined using the topological space $X$ itself. In particular, a point of a germ will be a point of $X$, and a sheaf on a germ will be a sheaf on $X$. Note that the underlying topological space of any germ $X$ is locally spectral, and and morphism of germs induces a spectral map between the underlying topological spaces. 

\begin{definition} A germ $X$ is said to be overconvergent if it admits a representative $(X,\bm{X})$ such that $X\subset \bm{X}$ is an overconvergent closed subset (i.e. is stable under generalisation).
\end{definition}

\subsubsection{}

It is perhaps worth carefully recalling the definitions of the different types of immersions for adic spaces and germs. Following \cite[(1.4.1)]{Hub96} a closed analytic subspace of an adic space $X$ is one defined by a coherent sheaf of ideals $\mathcal{I}\subset \mathcal{O}_X$. A morphism $f:X\rightarrow Y$ of adic spaces is a closed immersion if it is isomorphic to the inclusion of a closed analytic subspace, and a locally closed immersion if it factors as the composition of a closed immersion followed by an open immersion.

A (locally) closed immersion of germs is one that has a representative $f:(X,\bm{X}) \rightarrow (Y,\bm{Y})$ as a morphism of pairs such that $f:\bm{X}\rightarrow \bm{Y}$ is a locally closed immersion of adic spaces, and $X$ is (locally) closed in $Y$. Finally, an open immersion of germs is one that has representative $f:(X,\bm{X}) \rightarrow (Y,\bm{Y})$ a morphism of pairs such that $f:\bm{X}\rightarrow \bm{Y}$ is an open immersion of adic spaces, and $X$ is open in $Y$. Note that a locally closed immersion of germs, which is an open immersion on the underlying topological spaces, need not be an open immersion.

We will also use the following elementary fact constantly.

\begin{lemma} \label{lemma: open oc pp}Let $j:U\rightarrow X$ be an open immersion of germs. Then $j$ is partially proper if and only if $U$ is an overconvergent open subset of $X$.
\end{lemma}

\begin{proof}
This follows from the valuative criterion of properness \cite[Corollary 1.10.21]{Hub96}.
\end{proof}

\subsubsection{} As with adic spaces, or pseudo-adic spaces, fibre products in general are not representable in $\mathbf{Germ}$. However, they will be representable if at least one of the morphisms is locally of weakly finite type. If $X\overset{f}{\rightarrow} Z \leftarrow Y$ is a diagram of germs, represented by a diagram
\[ \xymatrix{  & (X,\bm{X}) \ar[d]^f \\ (Y,\bm{Y}) \ar[r] & (Z,\bm{Z}) } \]
of pairs, with $f$, say, locally of weakly finite type, then the fibre product $X\times_Y Z$ is represented by the pair
\[ (p_1^{-1}(X)\cap p_2^{-1}(Y),\bm{X} \times_{\bm{Z}} \bm{Y}).\]

\begin{example} Let $k^\circ$ be a complete discrete valuation ring, $k$ its fraction field, $\fr{P}$ a formal scheme, flat and of finite type over $k^\circ$, and $X\subset \fr{P}$ a locally closed subset. Set $\kappa=(k,k^\circ)$. Then, for any $n\geq 0$, we can consider $X$ as a locally closed subset of $\widehat{\A}^n_\fr{P}$ via the zero section, and we have
\[ ]X[_{\widehat{\A}^n_\fr{P}} \cong ]X[_\fr{P} \times_{\kappa} \D^n_{\kappa}(0;1^-)\]
as germs over $\kappa$.
\end{example}

If $X$ is a germ with ambient adic space $\bm{X}$ we will write $\mathcal{O}_X:=\mathcal{O}_{\bm{X}}\!\!\mid_X$. We can similarly extend the notions of local rings, residue fields, et cetera, to germs of spaces. For example, if $x\in X$ is a point of a germ, then we may speak of the local ring $\mathcal{O}_{X,x}$, the residue field $k(x)$, and the residue valuation ring $k(x)^+\subset k(x)$.

\subsection{Local germs}

We recall from \cite{Hub96} the definition of an (analytic) affinoid field. In the theory of analytic adic space, these play a role roughly analogous to that played by local rings in algebraic geometry.

\begin{definition} An affinoid field is an affinoid ring $\kappa=(k,k^+)$ where $k$ is a field, $k^+\subset k$ is a valuation ring, and the valuation topology on $k$ can be induced by a height $1$ valuation. We define the height of $\kappa$ to be the height of the valuation ring $k^+$.
\end{definition}

Note that the condition on the topology of $k$ is equivalent to requiring that $(k,k^+)$ is a Tate pair.

\begin{definition} An adic space is called \emph{local} if it is isomorphic to the spectrum $\spa{k,k^+}$ of an affinoid field. A germ is called local if it has a representative of the form $(X,\bm{X})$ with $\bm{X}$ a local adic space.
\end{definition}

Note that the points of a local germ are totally ordered by generalisation.

\begin{example} \begin{enumerate}
\item Let $X$ be an adic space, and $x\in X$ a point. Then $\kappa(x):=(k(x),k(x)^+)$ is an affinoid field, and $\spa{\kappa(x)}$ is a local adic space, called the localisation of $X$ at $x$. There is a canonical morphism  $\spa{\kappa(x)}\rightarrow X$ which induces a homeomorphism between $\spa{\kappa(x)}$ and the set $G(x)$ of generalisations of $x$.
\item We can also do the same with points of germs. Namely, if $x\in X\subset \bm{X}$ is such a point, then $\spa{\kappa(x)}\cap X$, which is naturally homeomorphic to the set of generalisations of $x$ within $X$, will be a closed subset of the local adic space $\spa{\kappa(x)}$. It is therefore a local germ.
\end{enumerate}
\end{example}

\subsection{}

Let $f:X\rightarrow Y$ be a morphism of germs, locally of weakly finite type. We may therefore take the fibre product of $f$ with any other morphism $g$. For any point $y\in Y$, we define
\[ X_{(y)}:= X\times_Y (\spa{\kappa(y)}\cap Y), \]
which is locally of weakly finite type over the local germ $\spa{\kappa(y)}\cap Y$. Note that the underlying topological space of $X_{(y)}$ is equal to $f^{-1}(G(y))$, and it contains the space $f^{-1}(y)$ as a closed subspace, equal to the closed fibre of the natural map
\[ X_{(y)} \rightarrow \spa{\kappa(y)}\cap Y. \]
The following is then just a rephrasing of Corollary \ref{cor: pushforward stalks}.

\begin{corollary} \label{cor: push 2} Let $f:X\rightarrow Y$ be a coherent morphism of germs, locally of weakly finite type, and $\mathscr{F}$ a sheaf on $X$. Then, for all $q\geq0$, the natural map
\[ (\mathbf{R}^qf_*\mathscr{F})_y \rightarrow {\rm H}^q(X_{(y)},\mathscr{F}|_{X_{(y)}})\]
is an isomorphism.
\end{corollary}

\section{Proper pushforwards on germs of adic spaces} \label{sec: defn}

We can now define proper pushforwards for adic spaces, following Huber. 

\subsection{Sections with proper support}

Let $f:X \rightarrow Y$ be a morphism of germs, separated and locally of $^+$weakly finite type. 
Let $\mathscr{F}$ be a sheaf on $X$, $V\subset Y$ an open subset and $s\in \Gamma(f^{-1}(V),\mathscr{F})$ a section. Then the support \[ \mathrm{supp}(s):= \left\{\left. x\in f^{-1}(V) \right\vert s_x \neq 0 \right\} \subset f^{-1}(V) \]
is a germ of an adic space (as it is a closed subset of $f^{-1}(V)$), and it therefore makes sense to ask whether or not the natural map $\mathrm{supp}(s)\rightarrow V$ is proper. 

\begin{definition} \label{defn: proper supports} Define
\[f_!\mathscr{F} \subset f_*\mathscr{F} \]
to be the subsheaf consisting of sections $s\in \Gamma(V,f_*\mathscr{F})=\Gamma(f^{-1}(V),\mathscr{F})$ whose support is proper over $V$. 
\end{definition}

We will sometimes denote ${\rm H}^0(Y,f_!(-))$ by either ${\rm H}^0_c(X/Y,-)$ or $\Gamma_c(X/Y,-)$. If $f:X\rightarrow Y$ is partially proper, then the support of $s\in \Gamma(V,f_*\mathscr{F})$ is proper over $V$ if and only if it is quasi-compact over $V$.

As a first example, we can show that this definition recovers the usual extension by zero functor for open immersions.

\begin{lemma} \label{lemma: open f! extension by 0} Suppose that $f$ is an open immersion. Then $f_!$ is isomorphic to the extension by zero functor.
\end{lemma} 

\begin{proof}
We clearly have $(f_!\mathscr{F})|_{X} \cong \mathscr{F}$, so it suffices to show that  $(f_!\mathscr{F})|_{Y\setminus X}=0$. Let $y\in Y\setminus X$, $y\in V\subset Y$ an open neighbourhood and $s \in \Gamma(V\cap X,\mathscr{F})$ a section whose support (considered as a closed subset of $V\cap X$) is proper over $V$. Then $\mathrm{supp}(s)$ must be a closed subset of $V$, which does not contain $y$. Hence there exists an open subset $y\in W\subset V$ such that $s|_{W\cap X}=0$, in other words $s=0$ in $(f_!\mathscr{F})_y$. Since $s$ was arbitrary, we see that $(f_!\mathscr{F})_y=0$, and since $y$ was arbitrary, we see that $(f_!\mathscr{F})|_{Y\setminus X}=0$.
\end{proof}

\subsection{Comparison with van\thinspace der\thinspace Put's definition} Whenever $f:X\rightarrow Y$ is a morphism of adic spaces of finite type over a discretely valued affinoid field $(k,k^\circ)$, and the base $Y$ is affinoid, a definition of ${\rm H}^0_c(X/Y,-)$ has already been given in \cite[Definitions 1.4]{vdP92}. In fact, van\thinspace der \thinspace Put worked with rigid analytic spaces rather than adic spaces, but since the underlying topoi are the same \cite[(1.1.11)]{Hub96} we can transport his definition to the adic context.

Recall that if $U,V\subset X$ are open affinoids, we write
\[ U\Subset_Y V\]
if there exists a closed immersion $V\rightarrow \D^n_Y(0;1)$ over $Y$ such that $U\subset \D^n_Y(0;1^-)$. 

\begin{definition}[van\thinspace der\thinspace Put] For any sheaf $\mathscr{F}$ on $X$, the subgroup
\[ {\rm H}^0_{c,\mathrm{vdP}}(X/Y,\mathscr{F})\subset {\rm H}^0(X,\mathscr{F}) \]
is defined to be
\[ \underset{U}{\mathrm{colim}}\, {\rm H}^0_{\overline{U}}(X,\mathscr{F}),  \]
where the colimit is over all finite unions $U$ of affinoids $U_i$ for which there exist affinoids $V_i\subset X$ such that $U_i\Subset_Y V_i$, $\overline{U}$ denotes the closure of $U$ in $X$, and ${\rm H}^0_{\overline{U}}$ denotes sections with support in the closed subset $\overline{U}\subset X$.
\end{definition}

\begin{lemma} Assume that $f:X\rightarrow Y$ is a partially proper morphism, locally of finite type between adic spaces, such that $Y$ is affinoid and of finite type over a discretely valued affinoid field. Then ${\rm H}^0_c(X/Y,-)={\rm H}^0_{c,\mathrm{vdP}}(X/Y,-)$ as subfunctors of ${\rm H}^0(X,-)$. 
\end{lemma}

\begin{proof}
Since any closed subset of $X$ is partially proper over $Y$, the support of a section $s\in {\rm H}^0(X,\mathscr{F})$ is proper over $Y$ if and only if it is quasi-compact over $Y$, if and only if it is quasi-compact. On the other hand, since $f$ is partially proper, it follows from \cite[Remark 1.3.19]{Hub96} that the collection of open affinoids $U\subset X$ for which there exists an open affinoid $V\subset X$ such that $U\Subset_Y V$ forms a basis for the topology of $X$. 

It therefore suffices to show that a closed subset of $X$ is quasi-compact if and only if it it is contained in the closure of the union of finitely many such open affinoids $U$. The `only if' direction is clear, for the `if' direction, we use the fact that $X$ is taut \cite[Lemmaa 5.1.3, 5.1.4]{Hub96}, and so the closure of any quasi-compact open in $X$ is quasi-compact. 
\end{proof}

\begin{remark} The result is false without some assumption on $f$. For example, if $Y=\spa{\kappa}$ with $\kappa=(k,k^\circ)$ an affinoid field of height one, and $X=\D^1_\kappa(0;1)$ is the closed unit disc, then ${\rm H}^0_c(X,-)$ is genuinely different from ${\rm H}^0_{c,\mathrm{vdP}}(X,-)$. In this case, van\thinspace der\thinspace Put's definition is equivalent to requiring sections to have support quasi-compact and disjoint from the closure of the Gauss point, whereas Definition \ref{defn: proper supports} only requires this support to be disjoint from the Gauss point itself. However, as we shall see, neither definition leads to a satisfactory theory in the non-partially proper case.
\end{remark}

\subsection{Basic properties of proper pushforwards}

The following properties of $f_!$ and $\Gamma_c(X/Y,-)$ can be verified exactly as in \cite[Proposition 5.2.2]{Hub96}. 

\begin{proposition}  Let $f:X\rightarrow Y$ be a morphism of germs, separated and locally
of $^+$weakly finite type.
\begin{enumerate}
\item The functors $\Gamma_c(X/Y,-)$ and $f_!$ are left exact.
\item The functor $f_!$ commutes with filtered colimits. If $Y$ is coherent, then so does $\Gamma_c(X/Y,-)$.
\item Let $g:Y\rightarrow Z$ be a morphism of germs, separated and locally of $^+$weakly finite type. Then the canonical identification $(g\circ f)_*= g_*\circ f_*$ induces $(g\circ f)_! = g_!\circ f_!$.
\end{enumerate}
\end{proposition}

\subsection{Derived proper pushforwards for partially proper morphisms}

For partially proper morphisms only, we now define $\mathbf{R}f_!$ as the derived functor of $f_!$. 

\begin{definition}  \label{defn: Rf! pp} Let $f:X\rightarrow Y$ be a partially proper morphism of germs. Define
\[ \mathbf{R}f_!:{\bf D}^+(X) \rightarrow {\bf D}^+(Y) \]
to be the total derived functor of $f_!$. For any $q\geq0$, define $\mathbf{R}^qf_! = \mathcal{H}^q(\mathbf{R}f_!)$.
\end{definition}

We will also write $\mathbf{R}\Gamma_c(X/Y,-)$ for the total derived functor of ${\rm H}^0_c(X/Y,-)$, and ${\rm H}^q_c(X/Y,-)$ for the cohomology groups of this complex.

\subsubsection{}

To show that these derived proper pushforwards compose correctly, we can relate them to ordinary pushforwards as in \cite[\S5.3]{Hub96}. 

\begin{lemma} \label{lemma: push key} Let $f:X\rightarrow Y$ be a partially proper morphism of germs, with $Y$ coherent. 
\begin{enumerate}
\item There exists a cover of $X$ by a cofiltered family of overconvergent open subsets $\{U_i\}_{i\in I}\subset X$, each of which has quasi-compact closure.
\item \label{num: push key 1 2} For any such family $\{U_i\}$, any sheaf $\mathscr{F}$ on $X$, and any $q\geq0$, there is a canonical isomorphism
\[ \mathrm{colim}_{i\in I} \mathbf{R}^qf_{i*}(j_{i!}\mathscr{F}|_{U_i})\isomto \mathbf{R}^qf_!\mathscr{F} \]
where $j_i:U_i\rightarrow \overline{U}_i$ denotes the canonical open immersion, and $f_i:\overline{U}_i\rightarrow Y$ the restriction of $f$.
\end{enumerate}
\end{lemma} 

The first claim was proved in \cite[Lemma 5.3.3]{Hub96}, and the second part is shown in exactly the same way as Huber does in the \'etale case, using the following lemma.

\begin{lemma} \label{lemma: push key 2} Let $X$ be a quasi-separated germ of an adic space, $U\subset X$ an overconvergent open subset with quasi-compact closure, and $j:U\hookrightarrow \overline{U}$ the natural inclusion. Then, for any flasque sheaf $\mathscr{I}$ on $X$, and any $q>0$, ${\rm H}^q(\overline{U},j_!(\mathscr{I}|_U))=0$.
\end{lemma}

\begin{proof}
Replacing $X$ by $\overline{U}$, we may assume that $\overline{U}=X$ and that $X$ is coherent. If we let $i:Z\rightarrow X$ denote the closed complement to $U$, then $Z$ is an overconvergent closed subset of $X$, then taking the long exact sequence in cohomology associated to using the exact sequence
\[ 0 \rightarrow j_!\mathscr{I}|_U \rightarrow \mathscr{I} \rightarrow i_*\mathscr{I}|_Z \rightarrow 0,\]
and using the fact that ${\rm H}^q(X,\mathscr{I})=0$ for $q>0$, it is enough to show that:
\begin{enumerate}
\item ${\rm H}^0(X,\mathscr{I})\rightarrow {\rm H}^0(Z,\mathscr{I}|_Z)$ is surjective;
\item ${\rm H}^q(Z,\mathscr{I}|_Z)=0$ for all $q>0$.
\end{enumerate}
Since $\mathscr{I}$ is flasque, and $Z$ is overconvergent, this follows from Lemma \ref{lemma: oc closed}. \end{proof}

\begin{proof}[Proof of Lemma \ref{lemma: push key}(\ref{num: push key 1 2})]
Define functors ${\rm T}^q:\mathbf{Sh}(X)\rightarrow \mathbf{Sh}(Y)$ by
\[{\rm T}^q(\mathscr{F})= \mathrm{colim}_{i\in I} \mathbf{R}^qf_{i*}(j_{i!}\mathscr{F}|_{U_i}).\]
The ${\rm T}^q$ form a $\delta$-functor, with ${\rm T}^0 = f_!$. Moreover, when $q>0$, we can deduce from Lemma \ref{lemma: push key 2} that ${\rm T}^q$ is effaceable, hence
\[ {\rm T}^q(\mathscr{F})\isomto \mathbf{R}^qf_!\mathscr{F} \]
as required.
\end{proof}

\begin{corollary} \label{cor: flasque pp acyclic} Let $f:X\rightarrow Y$ be partially proper. Then flasque sheaves on $X$ are $f_!$-acyclic.
\end{corollary}

The following two corollaries of Lemma \ref{lemma: push key} are proved word for word as in their \'etale counterparts \cite[Propositions 5.3.7, 5.3.8]{Hub96}.

\begin{corollary} \label{cor: colimits} Let $f:X\rightarrow Y$ be a partially proper morphism of germs. Then, for each $q\geq0$, the functor $\mathbf{R}^qf_!$ commutes with filtered colimits. If $Y$ is coherent, then so does ${\rm H}^q_c(X/Y,-)$. 
\end{corollary}

\begin{proof}
We may assume that $Y$ is coherent. Let $\{U_i\}_{i\in I}$, $j_i:U_i\rightarrow \overline{U}_i$ and $f_i:\overline{U}_i\rightarrow Y$ be as in Lemma \ref{lemma: push key}. Then $j_{i!}$ is a left adjoint, hence commutes with filtered colimits, and $f_i$ is coherent, hence $\mathbf{R}^qf_{i*}$ commutes with filtered colimits. Therefore
\[ \mathbf{R}^qf_! \cong \mathrm{colim}_{i\in I} \mathbf{R}^qf_{i*}(j_{i!}(-)|_{U_i})  \]
commutes with filtered colimits. The corresponding claim for ${\rm H}^q_c(X/Y,-)$ is proved similarly.
\end{proof}

\begin{corollary} \label{cor: comp part prop} Let $X\overset{f}{\rightarrow} Y \overset{g}{\rightarrow} Z$ be partially proper morphisms  of germs. Then there is a canonical isomorphism $\mathbf{R}(g\circ f)_!\cong \mathbf{R}g_!\circ \mathbf{R}f_!$ of functors ${\bf D}^+(X)\rightarrow {\bf D}^+(Z)$.
\end{corollary}

\begin{proof}
We may assume that $Z$ is coherent. By Corollary \ref{cor: flasque pp acyclic}, it suffices to show that for any injective sheaf $\mathscr{I}$ on $X$, the sheaf $f_!\mathscr{I}$ is flasque on $Y$. Thus we need to show that, for any quasi-compact open subset $W\subset Y$, the map 
\[ {\rm H}^0_c(X/Y,\mathscr{I}) \rightarrow {\rm H}^0_c(f^{-1}(W)/W,\mathscr{I})\]
is surjective. So pick a section $s\in {\rm H}^0(f^{-1}(W)/W,\mathscr{I})$ with proper support over $W$. Let $T$ be the closure of $\mathrm{supp}(s)$ in $X$, and let $s'\in {\rm H}^0(f^{-1}(W)\cup (X\setminus T),\mathscr{I})$ be the unique section with  $s'|_{f^{-1}(W)}=s$ and $s'|_{X\setminus T}=0$. Since $\mathscr{I}$ is flasque, we can pick $s''\in \Gamma(X,\mathscr{I})$ with $s''|_{f^{-1}(W)\cup X\setminus T}=s'$. Then $\mathrm{supp}(s'')\subset T$, and, since $f:X\rightarrow Y$ is taut \cite[Lemma 5.1.4 ii)]{Hub96}, $T$ is quasi-compact over $Y$. Thus $\mathrm{supp}(s'')$ is proper over $Y$, and gives a lift of $s$ to ${\rm H}^0_c(X/Y,\mathscr{I})$. 
\end{proof}

\subsection{Base change theorems}

We can also use Lemma \ref{lemma: push key} to describe the fibres of $\mathbf{R}f_!$, by combining it with Corollary \ref{cor: push 2}.

\begin{corollary} \label{cor: prop push supp weak} Let $f:X\rightarrow Y$ be a partially proper morphism of germs. Then for each $y\in Y$ the natural map
\[ (\mathbf{R}^qf_!\mathscr{F})_y \rightarrow {\rm H}^q_c(X_{(y)}/G(y),\mathscr{F}|_{X_{(y)}}) \]
is an isomorphism.
\end{corollary}

In particular, this says that whenever $y$ is a maximal point, the natural map
\[ (\mathbf{R}f_!\mathscr{F})_y  \rightarrow \mathbf{R}\Gamma_c(f^{-1}(y)/\spa{\kappa(y)},\mathscr{F}|_{f^{-1}(y)}) \]
is an isomorphism. This is not true in general if $y$ is not maximal, however, and we shall give a counter example in \S\ref{sec: counter examples} below. However, we do have the following base change result. 

\begin{lemma} \label{lemma: base change overconvergent closed} Let $f:X\rightarrow Y$ be a partially proper morphism of germs. Let $\mathscr{F}$ be a sheaf on $X$, and $Z\subset Y$ a locally closed subspace which is stable under generalisations. Let $f_Z: X_Z:= X\times_Y Z\rightarrow Z$ be the projection. Then the natural map
\[ (\mathbf{R}f_!\mathscr{F})|_Z \rightarrow \mathbf{R}f_{Z!}(\mathscr{F}|_{X_Z})\]
is an isomorphism. 
\end{lemma}

\begin{proof}
By Lemma \ref{lemma: push key} we may assume that $f$ is proper. In this case, since $Z$ is stable under generalisations, the result follows from Corollary \ref{cor: push 2}.
\end{proof}

\subsection{Cohomological amplitude}

If $f:X\rightarrow Y$ is a partially proper morphism between germs, then we have defined the functor
\[ \mathbf{R}f_!:{\bf D}^+(X)\rightarrow {\bf D}^+(Y). \]
If $X$ and $Y$ are moreover finite dimensional, then this will extend to a functor on the unbounded derived categories.

\begin{proposition} \label{prop: cd prop} Let $f:X\rightarrow Y$ be a partially proper morphism between finite dimensional germs. Then
\[ \mathbf{R}f_!: {\bf D}^+(X)\rightarrow {\bf D}^+(Y)\]
has cohomological amplitude contained in $[0,\dim X]$.
\end{proposition}

\begin{proof}
We may assume that $Y$ is coherent, and thus appeal to Lemma \ref{lemma: push key}. Choose open subsets $U_i\subset X$ as in the statement of the Lemma, with induced maps $f_i:\overline{U}_i\rightarrow Y$ and $j_i:U_i\rightarrow \overline{U}_i$. Since we have $\mathbf{R}f_! \cong \mathrm{colim}_{i\in I} \mathbf{R}f_{i*}j_{i!}$, it suffices to bound the cohomological dimension of $f_i$. But for $y\in Y$ we have $(\mathbf{R}f_{i*} \mathscr{F}_i)_y\cong \mathbf{R}\Gamma(\overline{U}_{i,(y)},\mathscr{F}_i)$, and the latter vanishes in cohomological degrees $\geq \dim \overline{U}_{i,(y)}$ by Theorem \ref{theo: cd ss}. It thus suffices to observe that $\dim \overline{U}_{i,(y)}\leq \dim X_{(y)}\leq \dim X$.  \end{proof}

\begin{corollary} \label{cor: cd prop} Let $f:X\rightarrow Y$ be a partially proper morphism between finite dimensional germs. Then the functor of proper pushforwards extends canonically to a functor
\[ \mathbf{R}f_!: {\bf D}(X)\rightarrow {\bf D}(Y)\]
on the unbounded derived categories. This functor sends ${\bf D}^-$ to ${\bf D}^-$ and ${\bf D}^b$ to ${\bf D}^b$. If $g:Y\rightarrow Z$ is another partially proper morphism, with $Z$ finite dimensional, then there is a canonical isomorphism
\[  \mathbf{R}g_!\circ \mathbf{R}f_!\isomto \mathbf{R}(g\circ f)_!\]
of functors
\[ {\bf D}(X)\rightarrow {\bf D}(Z).\]
\end{corollary}

\subsection{Mayer-Vietoris for proper pushforwards}

Let $f:X\rightarrow Y$ be a partially proper morphism between finite dimensional germs, and consider an open hypercover
\[ U_\bullet \rightarrow X \]
of $X$ by overconvergent open subsets. Thus each $U_n$ is partially proper over $Y$ by Lemma \ref{lemma: open oc pp}. Let $j_n:U_n\rightarrow X$ denote the given morphism (which is a disjoint union of the inclusion of overconvergent open subset of $X$), and $f_n:U_n \rightarrow Y$ the composition $f\circ j_n$. Suppose that we have a sheaf $\mathscr{F}$ on $X$. Then there is a resolution
\[ \ldots \rightarrow j_{1!}\mathscr{F}|_{U_1} \rightarrow j_{0!}\mathscr{F}|_{U_0} \rightarrow \mathscr{F} \rightarrow 0 \]
of $\mathscr{F}$, coming from the fact that $U_\bullet\rightarrow X$ is a hypercover. By Corollary \ref{cor: cd prop} we can apply $\mathbf{R}f_!$ to this resolution, and by Corollary \ref{cor: comp part prop} we know that $\mathbf{R}f_!\circ j_{n!}=\mathbf{R}f_{n!}$ (note that $j_{n!}$ is exact as it is the extension by zero along an open immersion). By Proposition \ref{prop: cd prop} the cohomological dimension of $\mathbf{R}f_{n!}$ is bounded independently of $n$, and we therefore obtain a convergent second quadrant spectral sequence 
 \[ E_{1}^{-n,q} = \mathbf{R}^qf_{n!}\mathscr{F}|_{U_{n}}  \Rightarrow \mathbf{R}^{-n+q}f_!\mathscr{F}\]
in the category $\mathbf{Sh}(Y)$ of abelian sheaves on $Y$. The terms $\mathbf{R}^qf_{n!}\mathscr{F}|_{U_{n}} $ can also be made slightly more explicit: if $U_n=\coprod_{m} U_{n,m}$ with each $U_{n,m}$ an open subset of $X$, and $f_{n,m}:U_{n,m}\rightarrow Y$ is the restriction of $f$ to $U_{n,m}$, then
\[ \mathbf{R}^qf_{n!}\mathscr{F}|_{U_{n}} = \bigoplus_{m} \mathbf{R}^qf_{n,m!}\mathscr{F}|_{U_{n,m}} \]
by Corollary \ref{cor: colimits}.

\begin{corollary} \label{cor: MV ss} Let $f:X\rightarrow Y$ be a partially proper morphism between finite dimensional germs, $\mathscr{F}$ a sheaf on $X$, and $U_\bullet\rightarrow X$ a hypercover by overconvergent opens, with $U_n=\coprod_{m} U_{n,m}$. Then, setting $f_{n,m}=f|_{U_{n,m}}$, there exists a convergent spectral sequence
\[ E_{1}^{-n,q} = \bigoplus_{m} \mathbf{R}^qf_{n,m!}\mathscr{F}|_{U_{n,m}} \Rightarrow \mathbf{R}^{-n+q}f_!\mathscr{F}\]
in the category $\mathbf{Sh}(Y)$ of abelian sheaves on $Y$.
\end{corollary}

\subsection{Module structures on proper pushforwards} 

Let $f:X\rightarrow Y$ be a partially proper morphism of germs, and suppose that we have sheaves of rings $\mathcal{A}_X$ and $\mathcal{A}_Y$ on $X$ and $Y$ respectively, together with a morphism $\mathcal{A}_Y\rightarrow f_*\mathcal{A}_X$ making $f$ into a morphism of ringed spaces. The principal example for us will of course be the structure sheaves $\mathcal{A}_X=\mathcal{O}_X$ and $\mathcal{A}_Y=\mathcal{O}_Y$.

\begin{lemma} If $\mathscr{I}$ is an injective $\mathcal{A}_X$-module, then $\mathscr{I}$ is $f_!$-acyclic.
\end{lemma}

\begin{proof}
Since injective $\mathcal{A}_X$-modules are flasque, this follows from Corollary \ref{cor: flasque pp acyclic}.
\end{proof}

\begin{corollary}[Projection formula] For any locally free $\mathcal{A}_{Y}$-module $\mathscr{E}$ of finite rank, there exists an isomorphism
\[ \mathscr{E}\otimes_{\mathcal{A}_Y}  \mathbf{R}f_! \mathscr{F} \cong \mathbf{R}f_! (f^*\mathscr{E}\otimes_{\mathcal{A}_X} \mathscr{F})  \]
in ${\bf D}^+(\mathcal{A}_Y)$.
\end{corollary}

\begin{remark} Here $f^*\mathscr{E}$ denotes module pullback $f^{-1}\mathscr{E}\otimes_{f^{-1}\mathcal{A}_Y} \mathcal{A}_Y$. Since $\mathscr{E}$ is locally free, the functors
\begin{align*}
\mathscr{E}\otimes_{\mathcal{A}_Y}(-) : &{\bf D}^+(\mathcal{A}_Y) \rightarrow {\bf D}^+(\mathcal{A}_Y) \\
f^*\mathscr{E}\otimes_{\mathcal{A}_X}(-) : &{\bf D}^+(\mathcal{A}_X) \rightarrow {\bf D}^+(\mathcal{A}_X)
\end{align*}
are well-defined.
\end{remark}

\begin{proof}
Let $\mathscr{I}^\bullet$ be an injective resolution of $\mathscr{F}$ as an $\mathcal{A}_X$-module. Then $f^*\mathscr{E}\otimes_{\mathcal{A}_X} \mathscr{I}^\bullet$ is an injective resolution of $f^*\mathscr{E}\otimes_{\mathcal{A}_X} \mathscr{F}$. We can therefore reduce to the case $\mathscr{F}$ injective, and we must produce a canonical isomorphism
\[ f_! (f^*\mathscr{E}\otimes_{\mathcal{A}_X} \mathscr{F}) \cong  \mathscr{E}\otimes_{\mathcal{A}_Y}  f_! \mathscr{F}.  \]
Note that both sides embed naturally into $ \mathscr{E}\otimes_{\mathcal{A}_Y} f_*\mathscr{F} \cong f_*(f^*\mathscr{E}\otimes_{\mathcal{A}_X} \mathscr{F})$, to check that the images are equal we may argue locally, and hence assume that $\mathscr{E}\cong \mathcal{A}_{Y}^{\oplus n}$. Since all functors in sight ($f^*,f_!,f_*,\otimes$) commute with finite direct sums, we may therefore reduce to the trivial case $\mathscr{E} = \mathcal{A}_{Y}$.
\end{proof}

\subsection{Comparison with separated quotients}

The next crucial result we will need is a comparison between $\mathbf{R}f_!$ as we have defined it here, and the classical notion of proper pushforwards for maps between locally compact topological spaces. To prepare for this, we note the following property of separated quotients.

\begin{proposition} \label{prop: sep proper} Let $X$ be a taut germ. Then $[X]$ is Hausdorff and locally compact, and the separation map $\mathrm{sep}_X$ is topologically proper.
\end{proposition}

\begin{proof}
This is \cite[Chapter 0, Proposition 2.5.5, Theorem 2.5.7, Corollary 2.5.9]{FK18}.
\end{proof}

Now, let $f:X\rightarrow Y$ be a partially proper morphism of germs. Then we have a diagram
\[ \xymatrix{ X \ar[r]^{\mathrm{sep}_X} \ar[d]_f & [X] \ar[d]^{[f]} \\ Y \ar[r]^{\mathrm{sep}_Y} & [Y] } \]
relating the separated quotients of $X$ and $Y$. If $Y$ is taut, then so is $X$ by \cite[Lemma 5.1.4 ii)]{Hub96}, and hence Proposition \ref{prop: sep proper} applies to both $X$ and $Y$. In this situation, we may consider the usual functor $[f]_!$ of sections whose support is topologically proper over $[Y]$, together with its total derived functor $\mathbf{R}[f]_!:{\bf D}^+([X])\rightarrow {\bf D}^+([Y])$. 

\begin{lemma} \label{lemma: support qc proper} Let $\mathscr{F}$ be a sheaf on $X$, $U\subset Y$ an overconvergent open subset, and let $s\in \Gamma(f^{-1}(U),\mathscr{F})$ a section. Write $[s]$ for $s$ considered as a section of $\Gamma([f]^{-1}([U]),\mathrm{sep}_{X*}\mathscr{F})$. Then $\mathrm{supp}(s)\rightarrow U$ is proper if and only if $\mathrm{supp}([s])\rightarrow [U]$ is topologically proper. 
\end{lemma}

\begin{remark} The hypothesis that $U$ is overconvergent is to ensure that $[U]$ is an open subset of $[Y]$.
\end{remark}

\begin{proof}
Note that since $f$ is partially proper, $\mathrm{supp}(s)\rightarrow U$ is proper if and only if it is quasi-compact. We first claim that $\mathrm{supp}([s])=\mathrm{sep}_X(\mathrm{supp}(s))$, which can be proved using \cite[Lemma 8.1.5]{Hub96}. Indeed, this shows that for any $x\in [X]$, we have $(\mathrm{sep}_{X*}\mathscr{F})_x \cong H^0(\overline{\{x\}},\mathscr{F}|_{\overline{\{x\}}})$, where the closure is taken inside $X$ (see also the proof of \cite[Proposition 8.1.4]{Hub96}). In particular, we see that $[s]_x=0$ if and only if $s|_{\overline{\{x\}}}=0$ if and only iff $s_y=0$ for all $y\in \overline{\{x\}}=\mathrm{sep}_X^{-1}(x)$. This implies that $\mathrm{supp}([s])=\mathrm{sep}_X(\mathrm{supp}(s))$ as claimed. 

Now, suppose that $\mathrm{supp}(s) \rightarrow U$ is quasi-compact, and that $K\subset [U]$ is quasi-compact. Since $[U]$ is Hausdorff, $K$ is closed, and hence the inverse image $\mathrm{sep}_Y^{-1}(K)$ is a closed, quasi-compact subset of $U$. 

In particular, $\mathrm{sep}_Y^{-1}(K)$ is contained inside a quasi-compact open subset $V\subset U$, whence the preimage $\mathrm{supp}(s)\cap f^{-1}(\mathrm{sep}_Y^{-1}(K))$ is a closed subset of the quasi-compact set $\mathrm{supp}(s)\cap f^{-1}(V)$, and is thus quasi-compact. Hence 
\[ \mathrm{supp}(s) \cap \mathrm{sep}_{X}^{-1}([f]^{-1}(K))=\mathrm{supp}(s) \cap f^{-1}(\mathrm{sep}_Y^{-1}(K))  \]
is quasi-compact, and so
\[ \mathrm{sep}_{X}(\mathrm{supp}(s) \cap \mathrm{sep}_{X}^{-1}([f]^{-1}(K)))=\mathrm{sep}_{X}(\mathrm{supp}(s)) \cap [f]^{-1}(K)=\mathrm{supp}([s]) \cap [f]^{-1}(K) \]
is quasi-compact. In other words, preimages of quasi-compact subsets under $\mathrm{supp}([s])\rightarrow [U]$ are quasi-compact. Since $[U]$ is locally compact, and $\mathrm{supp}([s])$ is Hausdorff, it follows that $\mathrm{supp}([s])\rightarrow [U]$ is topologically proper.

On the other hand, suppose that $\mathrm{supp}([s]) \rightarrow [U]$ is topologically proper, and let $V\subset U$ be a quasi-compact open subset. Then $[V]\subset [Y]$ is quasi-compact, and closed in $[Y]$. Thus $\overline{V}:=\mathrm{sep}_Y^{-1}([V])$ is quasi-compact and closed in $Y$ (and is in fact the closure of $V$ in $Y$, although we won't need that here). We know that $[f]^{-1}([V]) \cap \mathrm{supp}([s])$ is quasi-compact, and since $\mathrm{sep}_X$ is topologically proper, we see that 
\[ \mathrm{supp}(s)\cap f^{-1}(\overline{V})\subset \mathrm{sep}_{X}^{-1}([f]^{-1}([V]) \cap \mathrm{supp}([s]))\]
is contained in a quasi-compact closed subset of $f^{-1}(\overline{V})$, and is thus quasi-compact. 

Now, since $Y$ is quasi-separated and $V$ is quasi-compact, the inclusion $V\rightarrow Y$ is quasi-compact, and hence by \cite[Chapter 0, Corollary 2.1.6]{FK18} the morphism $V\rightarrow \overline{V}$ is quasi-compact. Thus by \cite[Lemma 1.10.7 c)]{Hub96} the morphism $f^{-1}(V)\cap \mathrm{supp}(s)\rightarrow f^{-1}(\overline{V})\cap \mathrm{supp}(s)$ is quasi-compact and hence $f^{-1}(V)\cap \mathrm{supp}(s)$ is quasi-compact. Since $V$ was arbitrary, we conclude that $\mathrm{supp}(s)\rightarrow U$ is quasi-compact as required.
\end{proof}

\begin{corollary} \label{cor: f! sep} Let $Y$ be a taut germ, and $f:X\rightarrow Y$ a partially proper morphism. Then there is a canonical isomorphism
\[ \mathbf{R}\mathrm{sep}_{Y*}\circ \mathbf{R}f_! \cong \mathbf{R}[f]_!\circ \mathbf{R}\mathrm{sep}_{X*}\]
of functors ${\bf D}^+(X)\rightarrow {\bf D}^+([Y])$.
\end{corollary}

\begin{proof}
Note that Lemma \ref{lemma: support qc proper} above gives rise to an equality
\[ \mathrm{sep}_{Y*}\circ f_! = [f]_!\circ \mathrm{sep}_{X*} \]
of subfunctors of $\mathrm{sep}_{Y*}\circ f_* = [f]_*\circ \mathrm{sep}_{X*}$, so we need to show that
\[ \mathbf{R}(\mathrm{sep}_{Y*}\circ f_!) \cong \mathbf{R}\mathrm{sep}_{Y*}\circ \mathbf{R}f_!\]
and
\[ \mathbf{R}([f]_!\circ \mathrm{sep}_{X*})\cong  \mathbf{R}[f]_!\circ \mathbf{R}\mathrm{sep}_{X*}.\]
The first follows from the proof of Corollary \ref{cor: comp part prop}, in particular the fact that $f_!$ sends injective sheaves to flasque sheaves. The second follows from the fact that $\mathrm{sep}_{X*}$ preserves injectives.
\end{proof}

\begin{corollary}  \label{cor: coh dim sep quot}Let $Y$ be a local germ, and $f\colon X\rightarrow Y$ a partially proper morphism with $\dim f=d$. Then
\[ {\rm H}^q_c([X],\mathscr{F})=0\]
for any sheaf $\mathscr{F}$ on $[X]$, and any $q>d$.
\end{corollary}

\begin{proof}
Replacing $Y$ by its maximal point doesn't change $[X]$, so we may assume that $Y$ consists of a single point. As in \cite[Proof of Proposition 8.1.4 i)]{Hub96} we see that $\mathscr{F}\isomto \mathbf{R}\mathrm{sep}_{X*}\mathrm{sep}_X^{-1}\mathscr{F}$, it therefore suffices to show that 
\[ {\rm H}^q_c(X/Y,\mathrm{sep}_X^{-1}\mathscr{F})=0 \]
for $q>d$. But now, applying Lemma \ref{lemma: push key}, this reduces to Theorem \ref{theo: cd ss}.
\end{proof}

\section{Kiehl partial properness and cohomological dimension}\label{sec: kpp}

In this section we will prove a result on the cohomological dimension of coherent sheaves for certain partially proper morphisms. The condition we require is in fact the original definition of partial properness given by Kiehl \cite{Kie67a}. Any such morphism has to be locally of finite type, which excludes many examples of partially proper morphisms (in particular, Huber's universal compactifications, constructed in \cite{Hub96}, are generally not locally of finite type). For morphisms locally of finite type, partial properness in the sense of Kiehl coincides with partial properness in many cases of interest, although it is still open whether or not the two coincide in general.

\subsection{Polydiscs and affine spaces over germs}

To begin with, we recall the definitions of polydiscs and affine spaces over a germ $Y$. To begin with, suppose that $Y=\spa{R,R^+}$ is a Tate affinoid adic space, then we have the usual definition
\[ \mathbb{D}^d_Y(0;1) = \spa{R\tate{\bm{z}},R^+\tate{\bm{z}}}\] 
of the closed unit polydisc over $Y$, using multi-index notation $\bm{z}=(z_1,\ldots,z_d)$. If $\varpi\in R$ is a quasi-uniformiser, and $q\in \Q_{\geq 0}$, we write $q=\frac{a}{b}$ in lowest terms, and define
\[ \mathbb{D}^d_Y(0;\norm{\varpi}^q) :=  \spa{\frac{R\tate{\bm{z},\bm{t}}}{(\varpi^a\bm{t}-\bm{z}^b)},\frac{R^+\tate{\bm{z},\bm{t}}}{(\varpi^a\bm{t}-\bm{z}^b)}}  \]
to be the ``closed disc of radius $\norm{\varpi}^{q}$''. Similarly, if $q<0$, we set $n=\lfloor q\rfloor$, and write $q-n=\frac{a}{b}$ in  lowest terms, and set
\[ \mathbb{D}^d_Y(0;\norm{\varpi}^q) :=  \spa{\frac{R\tate{\varpi^{-n}\bm{z},\bm{t}}}{(\varpi^a\bm{t}-(\varpi^{-n}\bm{z})^b)},\frac{R^+\tate{\varpi^{-n}\bm{z},\bm{t}}}{(\varpi^a\bm{t}-(\varpi^{-n}\bm{z})^b)}}.  \]
We then define
\[ \D^d_Y(0;1^-):=\bigcup_{n\geq 1} \D^d_Y(0;\norm{\varpi}^{\frac{1}{n}}),\;\;\;\;\A_Y^{d,\an}:=\bigcup_{n\geq 1} \D^d_Y(0;\norm{\varpi}^{-n}) \]
as well as analogous open discs
\[ \D^d_Y(0;\norm{\varpi}^{q-}):=\bigcup_{q' > q} \D^d_Y(0;\norm{\varpi}^{q'}) \]
``of radius $\norm{\varpi}^q$''. 

More generally, if $Y$ is an adic space admitting an element $\varpi\in \Gamma(Y,\mathcal{O}_Y)$ which is a quasi-uniformiser locally around every point $y\in Y$, then we can define any of 
\[ \D^d_Y(0;1),\;\;\mathbb{D}^d_Y(0;\norm{\varpi}^q),\;\;\D^d_Y(0;1^-),\;\;\A_Y^{d,\an},\;\;\D^d_Y(0;\norm{\varpi}^{q-})  \]
 by gluing. If $Y$ is a germ admitting a similar global quasi-uniformiser $\varpi\in\Gamma(Y,\mathcal{O}_Y)$, then, locally on some ambient adic space $\bm{Y}$, we can define analogous spaces over $Y$ by pulling back from those defined over $\bm{Y}$, for example, $\mathbb{D}^d_Y(0;\norm{\varpi}^q)$ is defined by the Cartesian diagram
\[ \xymatrix{ \mathbb{D}^d_Y(0;\norm{\varpi}^q) \ar[r] \ar[d] & \mathbb{D}^d_{\bm{Y}}(0;\norm{\varpi}^q) \ar[d] \\ Y\ar[r] & \bm{Y}. } \]
Finally, the definitions of $\D^d_Y(0;1), \D^d_Y(0;1^-)$ and $\A_Y^{d,\an}$ are independent of the choice of quasi-uniformiser, and hence the definition globalises to give $\D^d_Y(0;1), \D^d_Y(0;1^-)$ and $\A_Y^{d,\an}$ over an arbitrary germ $Y$. It is straightforward to check that if $Y'\rightarrow Y$ is any morphism of germs, then there is a natural Cartesian diagram
\[ \xymatrix{ \mathbb{D}^d_{Y'}(0;1) \ar[r] \ar[d] & \mathbb{D}^d_Y(0;1) \ar[d] \\ Y'\ar[r] & Y, } \]
as well as the obvious analogues for $\D^d_Y(0;1^-)$ and $\A_Y^{d,\an}$.

\begin{remark} Alternative constructions of $\D^d_Y(0;1),\A^{d,\an}_Y$ and $\D^d_Y(0;1^-)$, described in \cite{LS17}, are as fibre products
\begin{align*}
\D^d_Y(0;1) &= Y\times_{\spa{\Z,\Z}}\spa{\Z[\bm{z}],\Z[\bm{z}]} \\
\A^{d,\an}_Y &= Y\times_{\spa{\Z,\Z}}\spa{\Z[\bm{z}],\Z} \\
\D^d_Y(0;1^-) &= Y\times_{\spa{\Z,\Z}}\spa{\Z\pow{\bm{z}},\Z\pow{\bm{z}}}
\end{align*} 
in the category of (not necessarily analytic) adic spaces. Here $\Z$ and $\Z[\bm{z}]$ are given the discrete topology, and $\Z\pow{\bm{z}}$ the $\bm{z}$-adic topology. We will not use these constructions in this article. 
\end{remark}

\begin{lemma} Let $Y=\spa{R,R^+}$ be a Tate affinoid adic space, and $\varpi$ a quasi-uniformiser on $Y$. Then we have the following identifications of sets of sections:
\begin{align*}
\D^1_Y(0;1)(Y) &= R^+, \\
\D^1_Y(0;1^-)(Y) &= \left\{r\in R \mid \exists n\geq1\text{ s.t. } r^n\in \varpi R^+ \right\}=R^{\circ\circ}, \\
\A^{1,\an}_Y(Y) &= \left\{r\in R \mid \exists n\geq1\text{ s.t. } r\in \varpi^{-n}R^+ \right\} = R.
\end{align*}
\end{lemma}

\begin{proof}
Straightforward. 
\end{proof}

To define Kiehl's version of partial properness, we use the following result. 

\begin{proposition} \label{prop: kiehl pp} Let $f:X\rightarrow Y$ be a morphism of germs. The following conditions are equivalent. 
\begin{enumerate}
\item Locally on $X$ and $Y$, there exists a quasi-uniformiser $\varpi$ on $Y$, open covers $\{V_i\}_{i\in I}$ and $\{U_i\}_{i\in I}$ of $X$, integers $N_i\geq 1$, closed immersions $U_i\hookrightarrow \D_Y^{N_i}(0;1)$ over $Y$, and integers $m_i\geq 1$, such that $V_i\subset U_i\cap \D_Y^{N_i}(0;\norm{\varpi}^{\frac{1}{m_i}})$.
\item Locally on $X$ and $Y$, there exist open covers $\{V_i\}_{i\in I}$ and $\{U_i\}_{i\in I}$ of $X$, integers $N_i\geq 1$, and closed immersions $U_i\hookrightarrow \D_Y^{N_i}(0;1)$ over $Y$, such that $V_i\subset U_i\cap \D_Y^{N_i}(0;1^-)$.
\item Locally on $X$ and $Y$, there exists an open cover $\{U_i\}_{i\in I}$ of $X$, integers $N_i\geq 1$, and closed immersions $U_i\hookrightarrow \D_Y^{N_i}(0;1^-)$ over $Y$.
\end{enumerate}
\end{proposition}

\begin{proof}
Clearly we have (1)$\implies$(2) and (2)$\implies$(3), it therefore suffices to show that (3)$\implies$(1). So suppose that we have such a cover $U_i$. Localising on $Y$ we may choose a quasi-uniformiser $\varpi$ defined on an open neighbourhood of $Y$ in its ambient adic space. We define a new cover $\{U_{i,n}\}_{(i,n)\in I\times \N}$ of $X$ by $U_{i,n}:=U_i\cap \D^{N_i}_Y(0;\norm{\varpi}^{\frac{1}{n}})$. Each $U_{i,n}$ admits a closed immersion into $\D^{2N_i}_Y(0;1)$ defined informally by
\[ \bm{z} \in U_{i,n} \subset \D^{N_i}_Y(0;\norm{\varpi}^{\frac{1}{n}}) \mapsto (\bm{z},\varpi^{-1}\bm{z}^n) \in \D^{2N_i}_Y(0;1). \]
We now set $V_{i,n}$ to be $U_{i,n}\cap \D^{2N_i}_Y(0;\norm{\varpi}^{\frac{1}{n-1}})$, and we claim that the $V_{i,n}$ still cover $X$. But $V_{i,n}$ is defined in $U_{i,n}$ by the two equivalent conditions
\begin{align*}
v(\varpi^{-1}\bm{z}^{n-1})\leq 1 \\ 
v(\varpi^{-n}\bm{z}^{n(n-1)})\leq 1.
\end{align*}
Thus $V_{i,n}=U_{i,n-1}$, and so the $V_{i,n}$ do indeed cover $X$ as required.
\end{proof}

\begin{definition} \label{defn: pp Kiehl} We say that $f$ is partially proper in the sense of Kiehl if it is separated, taut, and satisfies the equivalent conditions of Proposition \ref{prop: kiehl pp}.
\end{definition}

\begin{remark} \begin{enumerate}
\item If $f$ is partially proper in the sense of Kiehl, then it is partially proper. 
\item If $f:X\rightarrow Y$ is partially proper in the sense of Kiehl, and $g:Z\rightarrow Y$ is any morphism, then $X\times_YZ \rightarrow Z$ is partially proper in the sense of Kiehl.
\item If $f$ and $g$ are partially proper in the sense of Kiehl, then so is $g\circ f$. If $g\circ f$ and $g$ are partially proper in the sense of Kiehl, then so is $f$. 
\item For any $Y$, the maps $\D^d_Y(0;1^-)\rightarrow Y$ and $\A^d_Y\rightarrow Y$ are partially proper in the sense of Kiehl. 
\item Any closed immersion is partially proper in the sense of Kiehl.
\item If $X$ and $Y$ are quasi-separated adic spaces locally of finite type over a discretely valued affinoid field, then any partially proper map $f:X\rightarrow Y$ is partially proper in the sense of Kiehl \cite[Remark 1.3.19]{Hub96}.
\item Any map which is partially proper in the sense of Kiehl is locally of finite type.
\end{enumerate} 
\end{remark}

The main result of this section is then the following. 

\begin{theorem} \label{theo: coh dim pp smooth}
Let $f:X\rightarrow Y$ be a morphism between finite dimensional adic spaces, partially proper in the sense of Kiehl, and set $d=\dim f$. If $\mathscr{F}$ is a coherent sheaf on $X$, then $\mathbf{R}^qf_!\mathscr{F}=0$ for $q> d$.
\end{theorem}

Note that Theorem \ref{theo: coh dim pp smooth} is stated only for adic spaces, not for more general germs. We will therefore be dealing with adic spaces until the end of \S\ref{sec: proof coh dim} below.

\subsection{Cohomology of coherent sheaves} 

Before embarking on the proof of Theorem \ref{theo: coh dim pp smooth}, we will need a couple of preliminary results on the cohomology of coherent sheaves on certain kinds of adic spaces. The first is the analogue of Theorems A and B for suitable `quasi-Stein' adic spaces.

\begin{proposition}\label{prop: A and B} Let $Y=\spa{R,R^+}$ be a Tate affinoid adic space, $X$ a closed analytic subspace of either $\D^N_Y(0;1^-)$ or $\A^{N,\an}_Y$, and $\mathscr{F}$ a coherent sheaf on $X$. Then $\mathscr{F}$ is generated by its global sections, and ${\rm H}^q(X,\mathscr{F})=0$ for all $q>0$.
\end{proposition}

\begin{remark} \label{rem: A and B} If $N=0$, i.e $X$ itself is Tate affinoid (and hence quasi-compact), it follows in the usual way that ${\rm H}^0(X,-)$ induces an equivalence of categories between coherent $\mathcal{O}_X$-modules and finitely generated ${\rm H}^0(X,\mathcal{O}_X)$-modules. In general, it seems reasonable to expect an analogue of the theory of `co-admissible modules' from \cite{ST03} to hold, although we did not think seriously about this question.
\end{remark}

\begin{proof}
When $N=0$, i.e. $X$ itself is a Tate affinoid, these claims follow from \cite[Chapter II, Theorems 6.5.7 and A.4.7]{FK18}. In general, we write
\[ X=\bigcup_q X\cap \D^N_Y(0;\norm{\varpi}^q)\]
for increasing $q$. Using the already proved case when $X$ is itself a Tate affinoid, it is then enough to show that
\[ \underset{q}{\mathrm{lim}}^{(1)}\; \Gamma(X\cap \D^N_Y(0;\norm{\varpi}^q),\mathscr{F})=0.\]
To do this, we apply \cite[Remarques 0.13.2.4, Proposition 0.13.2.2]{EGA3.1} and \cite[Chapter II, \S3.5, Theorem 1]{Bou98}. The facts needed to apply these results are the following: 
\begin{itemize}
\item each $\Gamma(X\cap \D^N_Y(0,\norm{\varpi}^{q}),\mathscr{F})$ has a canonical topology as a finitely generated module over the Banach ring $\Gamma(X\cap \D^N_Y(0,\norm{\varpi}^{q}), \mathcal{O}_X)$;
\item this topology is metrisable and complete;
\item each transition map 
\[ \Gamma(X\cap \D^N_Y(0,\norm{\varpi}^{q'}), \mathscr{F}) \rightarrow \Gamma(X\cap \D^N_Y(0,\norm{\varpi}^{q}), \mathscr{F}) \]
for $q'>q$ is uniformly continuous, and has dense image.
\end{itemize}
All of these can be easily verified.
\end{proof}

We will also need a slight generalisation of \cite[Proposition 1.3.6]{Ber93}, giving conditions for the structure sheaf to have vanishing higher direct images along the separation map.

\begin{definition} Let $X$ be a taut adic space. We say that $X$ is very good if every point $x\in X$ admits a Tate open affinoid neighbourhood $U$ such that $\overline{\{x\}}\subset U$.
\end{definition}

Thanks to \cite[Chapter 0, Corollary 2.3.31]{FK18}, this implies that $[U]$ contains an open neighbourhood of $\mathrm{sep}(x)$ in $[X]$. Also note that, by definition any very good adic space is necessarily taut.

\begin{proposition} \label{prop: acyclic sep F} Let $X$ be a very good adic space, $\mathrm{sep}\colon X\rightarrow [X]$ the separation map, and $\mathscr{F}$ a coherent $\mathcal{O}_X$-module. Then $\mathbf{R}^q\mathrm{sep}_{*}\mathscr{F}=0$ for $q>0$.
\end{proposition}

\begin{proof}
Let $x\in [X]$ be a maximal point, and choose a Tate open affinoid $U\subset X$ such that $\overline{\{x\}}\subset U\subset X$. It then follows from \cite[Corollary 0.2.3.31]{FK18} that $x\in \mathrm{int}_X(U)$ lies in the `overconvergent interior' of $U$, in other words, there exists an overconvergent open subset $x\in V\subset U$. Thus $x\in[V]\subset [U]$ is an open neighbourhood of $x\in [X]$ contained in $[U]$. Thus to prove that $\mathbf{R}^q\mathrm{sep}_{*}\mathscr{F}$ vanishes at $x$, we may replace $X$ by $U$, in other words we can assume that $X=\spa{R,R^+}$ is Tate affinoid, with $\varpi\in R$ a quasi-uniformiser.

In this case, by \cite[Chapter II, Proposition C.4.34]{FK18}, we can identify $[X]=\mathscr{M}(R)$ with the Berkovich spectrum of $R$. We now choose some $0<\rho<1$, and for every maximal point $x\in [X]$ we normalise $v_x:R\rightarrow \R_{\geq 0}$ so that $v_x(\varpi)=\rho$. Then, essentially by definition, $\mathscr{M}(R)$ has a basis of open subsets of the form
\[ U(f_1,\ldots,f_n;\lambda_1,\ldots,\lambda_n)=\left\{ x\in [X] \mid v_x(f_i)<\lambda_i \;\forall i \right\} \]
for $f_i\in R$ and $\lambda_i\in \R_{>0}$. Of course, it suffices to take $\lambda_i$ ranging over the dense subgroup $\rho^{\Q}\subset \R_{>0}$, and for a maximal point $x$, the condition $v_x(f_i)< \rho^{\frac{a}{b}}$ is equivalent to $v_x(\varpi^{-a}f^b)<1$. Thus $\mathscr{M}(R)$ in fact has a basis of open subsets of the form
\[ U(f_1,\ldots,f_n)=\left\{ x\in [X] \mid v_x(f_i)<1 \;\forall i \right\} \]
for $f_i\in R$. The preimage of $U(f_1,\ldots,f_n)$ in $\spa{R,R^+}$ therefore admits a closed immersion in the open unit polydisc $\D^n_X(0;1^-)$ over $X$, defined by
\[ x\in \mathrm{sep}^{-1}(U(f_1,\ldots,f_n))\mapsto (f_1(x),\ldots,f_n(x)).\]
It now follows from Proposition \ref{prop: A and B} that ${\rm H}^q(\mathrm{sep}^{-1}(U(f_1,\ldots,f_n)),\mathscr{F})=0$ for $q>0$, which completes the proof.
\end{proof}

\subsection{Proof of Theorem \ref{theo: coh dim pp smooth}} \label{sec: proof coh dim}

We now return to the proof of Theorem \ref{theo: coh dim pp smooth}, and there are two immediate reductions that we can make. First of all, we can assume that the base $Y$ is Tate affinoid, and secondly we can assume (by Corollary \ref{cor: MV ss}) that $X$ admits a closed immersion into some open unit polydisc $\D^N_Y(0;1^-)$. 

Moreover, using Proposition \ref{prop: kiehl pp} we can assume that $X=\D^N_Y(0;1^-)\cap Z$ for some closed immersion $Z \hookrightarrow  \D^N_Y(0;1)$, and that $\mathscr{F}$ extends to $Z$. This allows us to make one further reduction.

\begin{lemma} In proving Theorem \ref{theo: coh dim pp smooth}, we may assume that $\mathscr{F}=\mathcal{O}_X$.  
\end{lemma} 

\begin{proof}
Suppose that we know $\mathbf{R}^qf_!\mathcal{O}_X=0$ for all $q>d$. Since $\mathscr{F}$ extends to $Z$, and $Z$ is affonid, it follows from \cite[Chapter II, Theorems 6.5.7 and A.4.7]{FK18} as in the proof of Proposition \ref{prop: A and B} that there exists an exact sequence
\[ 0 \rightarrow \mathscr{F}_1 \rightarrow \mathcal{O}_X^{\oplus m} \rightarrow \mathscr{F}\rightarrow 0 \]
for some $m\geq 0$ and some coherent sheaf $\mathscr{F}_1$ extending to $Z$. We therefore deduce that $\mathbf{R}^qf_!\mathscr{F}\isomto \mathbf{R}^{q+1}f_!\mathscr{F}_1$ for all $q>d$. Repeating the argument, we find a coherent sheaf $\mathscr{F}_m$, extending to $Z$, such that $\mathbf{R}^qf_!\mathscr{F}\isomto \mathbf{R}^{q+m}f_!\mathscr{F}_m$ for all $q>d$. For $m$ large enough we have $\mathbf{R}^{q+m}f_!\mathscr{F}_m=0$ by Proposition \ref{prop: cd prop}, and hence $\mathbf{R}^qf_!\mathscr{F}=0$ as required. 
\end{proof}

Now, thanks to Corollary \ref{cor: prop push supp weak}, we have, for any $y\in Y$, an identification
\[ (\mathbf{R}^qf_!\mathcal{O}_X )_y \isomto {\rm H}^q_c( X_{(y)}/G(y),\mathcal{O}_X|_{X_{(y)}}). \]
If we let $\mathrm{sep}_{X_{(y)}}:X_{(y)}\rightarrow [X_{(y)}]$ denote the separation map, then Corollary \ref{cor: f! sep} gives
\[ {\rm H}^q_c( X_{(y)}/G(y),\mathcal{O}_X|_{X_{(y)}}) \isomto {\rm H}^q_c([X_{(y)}],\mathbf{R}\mathrm{sep}_{X_{(y)}*}(\mathcal{O}_X|_{X_{(y)}})).\]
Now applying Corollary \ref{cor: coh dim sep quot}, it suffices to show that $\mathbf{R}^q\mathrm{sep}_{X_{(y)}*}(\mathcal{O}_X|_{X_{(y)}})=0$ for $q>0$. Let $y\in U\subset Y$ be a Tate open affinoid neighbourhood of $y$, with preimage $f^{-1}(U)\subset X$ and separation map $\mathrm{sep}_{f^{-1}(U)}:f^{-1}(U)\rightarrow [f^{-1}(U)]$. 

\begin{lemma} The adic space $f^{-1}(U)$ is very good.
\end{lemma}

\begin{proof} Since $f^{-1}(U)$ is partially proper over an affinoid, it is taut. To prove that it is very good, we note that $U$ is Tate affinoid, so we may choose a quasi-uniformiser $\varpi$. Then $f^{-1}(U)$ is covered by the affinoid spaces $f^{-1}(U)\cap \D^N_U(0;\norm{\varpi}^{\frac{1}{n}})$ for $n\geq 1$. Now
\[ f^{-1}(U)\cap \D^N_U(0;\norm{\varpi}^{\frac{1}{n}}) \subset f^{-1}(U)\cap \D^N_U(0;\norm{\varpi}^{\frac{1}{n+1}-}) \subset f^{-1}(U)\cap \D^N_U(0;\norm{\varpi}^{\frac{1}{n+1}})\]
and each $f^{-1}(U)\cap \D^N_U(0;\norm{\varpi}^{\frac{1}{n+1}-})$ is an overconvergent open subset  of $f^{-1}(U)$. Thus, for every point $x\in f^{-1}(U)$, there is some $n$ such that
\[ \overline{\{x\}}\subset f^{-1}(U)\cap \D^N_U(0;\norm{\varpi}^{\frac{1}{n}-}) \subset f^{-1}(U)\cap \D^N_U(0;\norm{\varpi}^{\frac{1}{n}}), \]
and $f^{-1}(U)\cap \D^N_U(0;\norm{\varpi}^{\frac{1}{n}})$ is a Tate affinoid, since $U$ is.
\end{proof}

Thus Proposition \ref{prop: acyclic sep F} tells us that $\mathbf{R}^q\mathrm{sep}_{f^{-1}(U)*}(\mathcal{O}_X|_{f^{-1}(U)})=0$ for $q>0$, and Theorem \ref{theo: coh dim pp smooth} reduces to the following result. 

\begin{proposition} The natural map
\[ \mathrm{colim}_{y\in U\subset Y} \mathbf{R}\mathrm{sep}_{f^{-1}(U)*}(\mathcal{O}_X|_{f^{-1}(U)})|_{[X_{(y)}]} \rightarrow \mathbf{R}\mathrm{sep}_{X_{(y)}*}(\mathcal{O}_X|_{X_{(y)}}) \]
is an isomorphism. 
\end{proposition}

\begin{proof}
We compute the stalks on both sides at an arbitrary point $x\in [X_{(y)}]$. Note that any such point is a maximal point of $X$ (not just of $X_{(y)}$), and we see that $\overline{\{x\}}\cap X_{(y)}$ is the closure of $\{x\}$ inside $X_{(y)}$. Similarly, for any Tate open affinoid neighbourhood $y\in U\subset Y$ as above, $\overline{\{x\}}\cap f^{-1}(U)$ is the closure of $\{x\}$ inside $f^{-1}(U)$. If $Z$ is any taut analytic adic space, $\mathscr{F}$ any sheaf on $Z$, and $z\in [Z]$, then thanks to \cite[Lemma 8.1.5]{Hub96} (see also the proof of \cite[Lemma 8.1.4]{Hub96}), we have the natural base change isomorphism
\[ ({\bf R}^q{\rm sep}_*\mathscr{F})_z  \cong {\rm H}^q(\overline{\{z\}},\mathscr{F}).\]
Hence by Proposition \ref{prop: cohom limits} we can compute
\begin{align*}
\mathrm{colim}_{y\in U\subset Y} \mathbf{R}\mathrm{sep}_{f^{-1}(U)*}(\mathcal{O}_X|_{f^{-1}(U)}))_x &=  \mathrm{colim}_{y\in U\subset Y} \mathbf{R}\Gamma(\overline{\{x\}}\cap f^{-1}(U),\mathcal{O}_X)\\
&= \mathbf{R}\Gamma(\overline{\{x\}}\cap \bigcap_{y\in U\subset Y }f^{-1}(U),\mathcal{O}_X)\\
&=  \mathbf{R}\Gamma(\overline{\{x\}} \cap X_{(y)},\mathcal{O}_X) \\
&= \mathbf{R}\mathrm{sep}_{X_{(y)}*}(\mathcal{O}_X|_{X_{(y)}})_x
\end{align*}
as required.
\end{proof}

\subsection{The case of overconvergent germs}

We do not know whether Theorem \ref{theo: coh dim pp smooth} holds if $Y$ is replaced by an arbitrary germ. We do at least have the following special case.

\begin{corollary} \label{cor: coh dim omega germs}Let $f:X\rightarrow Y$ be a morphism betweem finite dimensional germs, partially proper in the sense of Kiehl, and smooth of relative dimension $d$. Let $\mathscr{F}$ be a coherent $\mathcal{O}_X$-module which extends to a coherent sheaf on some ambient adic space for $X$. Then $\mathbf{R}^qf_!\mathscr{F}=0$ for all $q > d$.
\end{corollary}

\begin{remark} \begin{enumerate}
\item  Recall that a germ is overconvergent if it is stable under generalisation inside its ambient adic space. 
\item It is possible that the hypothesis that $\mathscr{F}$ extends to some neighbourhood of $X$ is automatically satisfied. This will certainly be the case in the situation of Example \ref{exa: germ}(\ref{num: exa germ 2}).
\end{enumerate}
\end{remark}

\begin{proof}
Choose an ambient adic space $\bm{Y}$ of $Y$. By localising on $\bm{Y}$ we may assume that it is Tate affinoid, with quasi-uniformiser $\varpi\in \Gamma(\bm{Y},\cO_{\bm{Y}})$. By Corollary \ref{cor: MV ss} we may assume that $X$ admits a closed immersion $u:X\hookrightarrow \D^N_Y(0;1^-)$ for some $n$. We can therefore extend $f$ to a diagram of pairs
\[ \xymatrix{ (X,\bm{X}) \ar[r]^-{u} \ar[dr]_f & (\D^N_Y(0;1^-),\D^N_{\bm{Y}}(0;1^-)) \ar[d]^{\pi} \\ & (Y,\bm{Y})   } \]
such that:
\begin{itemize}
\item $f\colon \bm{X}\rightarrow \bm{Y}$ is smooth, and $X=f^{-1}(Y)$;
\item $\mathscr{F}$ extends to a coherent sheaf on $\bm{X}$;
\item $u \colon\bm{X}\rightarrow \D^N_{\bm{Y}}(0;1^-)$ is a locally closed immersion.
\end{itemize}
Note that $X=\pi^{-1}(Y)\cap \bm{X}$ as subspaces of $\D^N_{\bm{Y}}(0;1^-)$. Let $\bm{U}\subset \D^N_{\bm{Y}}(0;1^-)$ be open subspace such that $\bm{X}$ is a closed analytic subspace of $\bm{U}$.

Since $\bm{X}$ is a locally closed analytic subspace of $\D^N_{\bm{Y}}(0;1^-)$, it is closed under generalisations, and since $Y$ is an overconvergent closed subset of $\bm{Y}$, it follows that $\pi^{-1}(Y)$ is an overconvergent closed subset of $\D^N_{\bm{Y}}(0;1^-)$. Hence $X=\pi^{-1}(Y)\cap \bm{X}$ is closed under generalisations inside $\D^N_{\bm{Y}}(0;1^-)$, that is, it is an overconvergent closed subset of $\D^N_{\bm{Y}}(0;1^-)$. It therefore follows from \cite[Chapter 0, Proposition 2.3.17]{FK18} that each $X\cap \D^N_{\bm{Y}}(0;\norm{\varpi}^{\frac{1}{n}})$ admits a basis of neighbourhoods in $\D^N_{\bm{Y}}(0;\norm{\varpi}^{\frac{1}{n}})$ consisting of overconvergent open subsets. In particular there exist overconvergent open subsets $V_n\subset\D^N_{\bm{Y}}(0;\norm{\varpi}^{\frac{1}{n}})$ such that
\[ X\cap \D^N_{\bm{Y}}(0;\norm{\varpi}^{\frac{1}{n}}) \subset V_n\subset \bm{U}\cap \D^N_{\bm{Y}}(0;\norm{\varpi}^{\frac{1}{n}}) . \]
Since $V_n$ is overconvergent, it is the preimage of an open subset of $[\D^N_{\bm{Y}}(0;\norm{\varpi}^{\frac{1}{n}})]$ via the separation map. Thus arguing as in the proof of Proposition \ref{prop: acyclic sep F}, we see that $V_n$ can be covered by open subsets $V_{n,i}$ admitting closed immersions into $\D^N_{\bm{Y}}(0;\norm{\varpi}^{\frac{1}{n}})\times_{\bm{Y}} \D^{M_{n,i}}_{\bm{Y}}(0;1^-)$ over $\D^N_{\bm{Y}}(0;\norm{\varpi}^{\frac{1}{n}})$. It follows that
\[ V_n^-:=V_n\cap \D^{N}_{\bm{Y}}(0;\norm{\varpi}^{\frac{1}{n}-}) \] admits a covering by open subsets $V^-_{n,i}:=V_{n,i}\cap \D^{N}_{\bm{Y}}(0;\norm{\varpi}^{\frac{1}{n}-})$, each of which admits a closed immersion into $\D^{2N+M_{n,i}}_{\bm{Y}}(0;1^-)$ over $\bm{Y}$. Thus the $V^-_{n,i}$ are a collection of open subspaces of $\D^N_{\bm{Y}}(0;1^-)$, covering $X$, and there are closed immersions
\[ \bm{X}\cap V^-_{n,i}\rightarrow V^-_{n,i} \rightarrow \D_{\bm{Y}}^{2N+M_{n,i}}(0;1^-)  \]
of adic spaces over $\bm{Y}$. 

Therefore, by localising on $X$, and once more appealing to Corollary \ref{cor: MV ss}, we may reduce to the case that $f$ extends to a diagram of pairs
\[ \xymatrix{ (X,\bm{X}) \ar[r]^-{u} \ar[dr]_f & (\D^{N}_Y(0;1^-),\D^{N}_{\bm{Y}}(0;1^-)) \ar[d]^{\pi} \\ & (Y,\bm{Y})   } \]
such that $u:\bm{X}\rightarrow \D^N_{\bm{Y}}(0;1^-)$ is a \emph{closed} immersion, and $\mathscr{F}$ extends to $\bm{X}$. In this case, since $Y\subset \bm{Y}$ is overconvergent, and $X=f^{-1}(Y)$ inside $\bm{X}$ by smoothness of $f$, we can combine Lemma \ref{lemma: base change overconvergent closed} with Theorem \ref{theo: coh dim pp smooth} to conclude.
\end{proof}

\section{The trace map} \label{sec: trace}

In this section, we construct a trace map for the class of smooth morphisms which are partially proper in the sense of Kiehl, and whose target is an overconvergent and finite dimensional germ. This is a morphism
\[ \mathrm{Tr}_{X/Y}: \mathbf{R}f_!\Omega^{\bullet}_{X/Y}[2d] \rightarrow \mathcal{O}_Y \]
in the derived category of $\cO_Y$-modules, satisfying the conditions outlined in the introduction. We closely follow the argument of \cite{vdP92}, see also \cite{Bey97,Chi90}.

\subsection{The relative open unit polydisc}  \label{sec: relative open unit disc}

We first construct a trace map when $X=\D^d_Y(0;1^-)$ is the relative open unit polydisc over a Tate affinoid adic space $Y=\spa{R,R^+}$. Choose a quasi-uniformiser $\varpi\in R^\times\cap R^{\circ\circ}$. Since $X=\D^d_Y(0;1^-)$ is partially proper over $Y$, the support of a section of some sheaf $\mathscr{F}$ on $X$ is proper over $Y$ if and only if it is quasi-compact over $Y$, if and only if it is quasi-compact. The closure $\overline{\D}_n$ of $\D_Y^d(0;\norm{\varpi}^{\frac{1}{n}})$ inside $\D^d_Y(0;1^-)$ is quasi-compact, and moreover any quasi-compact subset of $\D^d_Y(0;1^-)$ has to be contained in $\overline{\D}_n$ for some $n$. Thus, if we let ${\rm H}^q_Z(X,-)$ denote cohomology groups with support in a closed subset $Z\subset X$, we find that
\[ {\rm H}^q_c(\D^d_Y(0;1^-)/Y,\mathscr{F}) = \mathrm{colim}_n {\rm H}^q_{\overline{\D}_n}(\D^d_Y(0;1^-),\mathscr{F}),\]
for any sheaf $\mathscr{F}$ on $\D^d_Y(0;1^-)$.

Using Proposition \ref{prop: A and B}, we can see that ${\rm H}^q(\D^d_Y(0;1^-),\mathscr{F})=0$ for any coherent $\mathcal{O}_{\D^d_Y(0;1^-)}$-module $\mathscr{F}$, and any $q>0$. Thus we deduce isomorphisms
\[ {\rm H}^q_c(\D^d_Y(0;1^-)/Y,\mathscr{F}) \isomto \begin{cases} \mathrm{ker}\left( {\rm H}^0(\D^d_Y(0;1^-),\mathscr{F}) \rightarrow \mathrm{colim}_n {\rm H}^0(\D^d_Y(0;1^-)\setminus \overline{\D}_n,\mathscr{F}) \right)   & q=0 \\ \mathrm{coker}\left( {\rm H}^0(\D^d_Y(0;1^-),\mathscr{F}) \rightarrow \mathrm{colim}_n {\rm H}^0(\D^d_Y(0;1^-)\setminus \overline{\D}_n,\mathscr{F}) \right)   & q=1
 \\ \mathrm{colim}_n {\rm H}^{q-1}(\D^d_Y(0;1^-)\setminus \overline{\D}_n,\mathscr{F}) & q>1. \end{cases} \]
We can cover $\D^d_Y(0;1^-)\setminus \overline{\D}_n$ by the spaces
\[ U_{i,n}:= \left\{\left. x\in \D^d_Y(0;1^-) \right\vert v_{[x]}(\varpi^{-1}z^n_i) > 1  \right\}, \]
each of which admits a closed immersion into an open polydisc over $Y$. Again, Proposition \ref{prop: A and B} therefore implies that coherent sheaves have vanishing higher cohomology groups on each $U_{i,n}$. The same reasoning applies to all intersections $\cap_{i\in I} U_{i,n}$, so we can compute the cohomology of $\mathscr{F}$ on $\D^d_Y(0;1^-) \setminus \overline{\D}_n$ as the cohomology of the \v{C}ech complex
\[ \bigoplus_{i=1}^d {\rm H}^0(U_{i,n},\mathscr{F}) \rightarrow \bigoplus_{i<j} {\rm H}^0(U_{i,n}\cap U_{j,n},\mathscr{F}) \rightarrow \ldots \rightarrow \bigoplus_{i=1}^d {\rm H}^0(\cap_{j\neq i} U_{j,n},\mathscr{F} ) \rightarrow  {\rm H}^0(\cap_i U_{i,n},\mathscr{F} ). \]
In the particular case when $\mathscr{F}=\omega_{\D^d_Y(0;1^-)/Y}$, we can therefore give a complete description of the cohomology groups ${\rm H}^q_c(\D^d_Y(0;1^-)/Y,\mathscr{F})$ as follows. Choose co-ordinates $z_1,\ldots,z_d$ on $Y$, and let $R\weak{z_1^{-1},\ldots,z_d^{-1}}$ denote the set of overconvergent series in $z_1^{-1},\ldots,z_d^{-1}$, that is, series of the form
\[ \sum_{i_1,\ldots,i_d\leq 0 } r_{i_1,\ldots,i_d}z_1^{i_1}\ldots z_d^{i_d},\;\;\;\;r_{i_1,\ldots,i_d}\in R, \]
for which there exists $n\geq 1$ such that $r_{i_1,\ldots,i_d}^n\varpi^{i_1+\ldots+i_d}\rightarrow 0$ as $(i_1,\ldots,i_d)\rightarrow -\infty$. Then
\[ {\rm H}^q_c(\D^d_Y(0;1^-)/Y,\omega_{\D^d_Y(0;1^-)/Y})= \begin{cases}
R\weak{z_1^{-1},\ldots,z_d^{-1}}\cdot {\rm d}\!\log z_1 \wedge\ldots\wedge {\rm d}\!\log z_d  & q=d \\ 0 & q\neq d.
\end{cases}  \]

We can therefore define the trace map
\begin{align*}
\Tr_{z_1,\ldots,z_d}: {\rm H}^d_c(\D^d_Y(0;1^-)/Y,\omega_{\D^d_Y(0;1^-)/Y}) &\rightarrow {\rm H}^0(Y,\mathcal{O}_Y) \\
\sum_{i_1,\ldots,i_d\leq 0 } a_{i_1,\ldots,i_d}z_1^{i_1}\ldots z_d^{i_d} \,{\rm d}\!\log z_1 \wedge \ldots {\rm d}\!\log z_d & \mapsto a_{0,\ldots,0}
\end{align*}
as in \cite[\S2.4]{vdP92} or \cite[\S2.1]{Bey97}. We can then globalise this construction to define
\[ \Tr_{z_1,\ldots,z_d}: \mathbf{R}^df_!\omega_{\D^d_Y(0;1^-)/Y} \rightarrow \mathcal{O}_Y \]
whenever the base $Y$ is an adic space. When $Y$ is an overconvergent germ, we pullback to $Y$ from its ambient adic space $\bm{Y}$ using Lemma \ref{lemma: base change overconvergent closed}. Also note that by Corollary \ref{cor: coh dim omega germs} we may view the trace map as a morphism
\[  \mathbf{R}f_!\omega_{\D^d_Y(0;1^-)/Y}[d]\rightarrow \mathcal{O}_Y \]
in ${\bf D}^b(\mathcal{O}_Y)$. The verification of the following is straightforward.
 
\begin{proposition} \label{prop: tr basic} Let $Y$ be an overconvergent germ. \begin{enumerate}
\item The trace map
\[ \Tr_{z_1,\ldots,z_d}: \mathbf{R}^df_!\omega_{\D^{d}_Y(0;1^-)/Y} \rightarrow \mathcal{O}_Y \]
vanishes on the image of $\mathbf{R}^df_!\Omega^{d-1}_{\D^{d}_Y(0;1^-)/Y}$, and hence induces a map
\[\Tr_{z_1,\ldots,z_d}: \mathbf{R}f_!\Omega^\bullet_{\D^{d}_Y(0;1^-)/Y}[2d] \rightarrow \mathcal{O}_Y.  \]
This map is an isomorphism.
\item The trace map is compatible with composition in the following sense: let $(z_1,\ldots,z_d)$ be co-ordinates on $\D^{d}_Y(0;1^-)$, let $1\leq e\leq d$, and let $h:\D^{d}_Y(0;1^-) \rightarrow \D^{e}_Y(0;1^-)$ be the projection $(z_1,\ldots,z_d)\mapsto (z_1,\ldots,z_e)$. Let $f:\D^{d}_Y(0;1^-)\rightarrow Y $ and $g: \D^{e}_Y(0;1^-)\rightarrow Y$ be the canonical identification
\[ \omega_{\D^{d}_Y(0;1^-)/Y} = h^*\omega_{\D^{e}_Y(0;1^-)/Y} \otimes \omega_{\D^{d}_Y(0;1^-)/\D^{e}_Y(0;1^-)}, \]
and the resulting identification
\[ \mathbf{R}f_!\left(\omega_{\D^{d}_Y(0;1^-)/Y} \right)[d] = \mathbf{R}g_!\left(\omega_{\D^{e}_Y(0;1^-)/Y} \otimes \mathbf{R}h_!\omega_{\D^{d}_Y(0;1^-)/\D^{e}_Y(0;1^-)}[d-e] \right)[e],  \]
we have
\[ \Tr_{z_1,\ldots,z_d} = \Tr_{z_1,\ldots,z_e} \circ \mathbf{R}g_!\left(\mathrm{id} \otimes \Tr_{z_{e+1},\ldots,z_d}\right).\]
\end{enumerate} 
\end{proposition}

We will see later on that $\Tr_{z_1,\ldots,z_d}$ is independent of the choice of co-ordinates $z_1,\ldots,z_d$; for now we record a special case of this.

\begin{lemma} \label{lemma: basic change 2} Suppose that $Y$ is an adic space, and let $z_1',\ldots,z'_d$ be a second set of co-ordinates on $\D^d_Y(0;1^-)$ defined by
\[ z_1' = z_1,\;\ldots\;,\; z_e'=z_e,\;z'_{e+1}= z_{e+1} +w_{e+1}, \;\ldots\;, \;z'_d=z_d + w_d  \]
for sections $w_i:\D^e_Y(0;1^-) \rightarrow \D^d_Y(0;1^-)$ of the natural projection. Then $\Tr_{z_1,\ldots,z_d}=\Tr_{z_1',\ldots,z'_d}$.
\end{lemma}

\begin{proof}
We may assume by localising that $Y=\spa{R,R^+}$ is affinoid, by induction that $e=d-1$, and by compatibility of the trace map with composition that $d=1$. In this case, the claim follows from the usual explicit calculation, which is an easy generalisation of a very special case of \cite[Proposition 2.1.3]{Bey97}.
\end{proof}

\begin{remark} \label{rem: A^d trace} 
As a variant, we can replace $f:\D^d_Y(0;1^-)\rightarrow Y$ everywhere by the relative analytic affine space $f:\A^{d,\an}_Y\rightarrow Y$. The construction of the trace map
\[ \Tr_{z_1,\ldots,z_d}: \mathbf{R}f_!\omega_{\A^{d,\an}_Y/Y}[	d] \rightarrow \mathcal{O}_Y, \]
is entirely similar, and the analogues of Proposition \ref{prop: tr basic} and Lemma \ref{lemma: basic change 2} hold. 
\end{remark}

\subsection{Duality for regular immersions} \label{sec: closed dual}

To extend the trace map from open polydiscs to more general morphisms, we will need a form of duality for regular closed immersions. Luckily, this follows quite quickly from the scheme-theoretic case.

\begin{lemma} \label{lemma: perf}Let $X=\spa{R,R^+}$ be a Tate affinoid adic space. Then
\[ \mathbf{R}\Gamma(X,-) : {\bf D}(\mathcal{O}_X) \rightarrow {\bf D}(R) \]
induces a $t$-exact equivalence of triangulated categories
\[ {\bf D}^+_\mathrm{coh}(\mathcal{O}_X) \isomto {\bf D}^+_\mathrm{coh}(R) \]
compatible with internal homs.
\end{lemma}

\begin{remark} The $t$-exactness here refers to the obvious $t$-structures on either side.
\end{remark}

\begin{proof} As noted in Remark \ref{rem: A and B}, ${\rm H}^0(X,-)$ is an equivalence of categories between coherent $\mathcal{O}_X$-modules and coherent (i.e. finitely generated) $R$-modules, and ${\rm H}^q(X,\mathscr{F})=0$ for any coherent $\mathcal{O}_{X}$-module $\mathscr{F}$ and any $q>0$. It then follows from this that 
\[ \mathbf{R}\Gamma: {\bf D}^+_\mathrm{coh}(\mathcal{O}_X) \rightarrow {\bf D}^+_\mathrm{coh}(R)  \]
is $t$-exact. To see that it is an equivalence, we consider the left adjoint
\[  -\otimes^\mathbf{L}_{R} \mathcal{O}_X :{\bf D}^+_\mathrm{coh}(R) \rightarrow {\bf D}^+_\mathrm{coh}(\mathcal{O}_X). \]
Essential surjectivity now follows from the fact that $\mathcal{O}_X$ is $R$-flat, and full faithfulness from the fact that the adjunction map
\[ \mathcal{O}_X \otimes^\mathbf{L}_R \mathbf{R}\Gamma(X,\mathscr{F})   \rightarrow \mathscr{F} \]
is an isomorphism, for any $\mathscr{F}\in {\bf D}^+_\mathrm{coh}(\mathcal{O}_X)$. Compatibility with internal homs now follows from the fact that the left adjoint $-\otimes^\mathbf{L}_R \mathcal{O}_X$ is monoidal.
\end{proof}

Recall that on a locally ringed space $(X,\mathcal{O}_X)$, a perfect complex of $\mathcal{O}_X$-modules is one that is locally quasi-isomorphic to a bounded complex of finite free $\mathcal{O}_X$-modules. Similarly, if $A$ is a ring, then a perfect complex of $A$-modules is a complex quasi-isomorphic to a bounded complex of finite projective $A$-modules.\footnote{Thus being a perfect complex of $A$-modules is stronger than being a perfect complex of $\mathcal{O}_{\spec{A}}$-modules.} The categories of such objects are viewed as full subcategories of ${\bf D}(\mathcal{O}_X)$ and ${\bf D}(A)$ respectively.

\begin{definition} A closed immersion $u:X\rightarrow Y$ of adic spaces is called regular of codimension $c$ if it is locally the vanishing locus of a regular sequence $f_1,\ldots,f_c\in \Gamma(Y,\mathcal{O}_Y)$.
\end{definition}

\begin{lemma} \label{lemma: closed dual} Let $u:X\rightarrow Y$ be a closed immersion of adic spaces, regular of codimension $c$, and let $\fr{n}_{X/Y}$ be the determinant of the normal bundle of $X$ in $Y$. Then, for any perfect complex $\mathscr{F}$ of $\mathcal{O}_X$-modules, there is a canonical isomorphism
\[ \Tr_u : u_*\mathbf{R}\Hom_{\mathcal{O}_X}(\mathscr{F},\mathfrak{n}_{X/Y}) \isomto \mathbf{R}\Hom_{\mathcal{O}_Y}(u_*\mathscr{F},\mathcal{O}_Y)[c] \]
in ${\bf D}(\mathcal{O}_Y)$, natural in $\mathscr{F}$. This is compatible with composition, in the sense that if $v:Y\rightarrow Z$ is a regular closed immersion of codimension $d$, and $\mathfrak{n}_{Y/Z}$ (resp. $\mathfrak{n}_{X/Z}$) the determinant of its normal bundle (resp. the normal bundle of $X$ in $Z$), then, via the identification $\fr{n}_{X/Z}= \fr{n}_{X/Y}\otimes_{\mathcal{O}_X} u^*\fr{n}_{Y/Z}$, the diagram
\[ \xymatrix{ (v\circ u)_*\mathbf{R}\Hom_{\mathcal{O}_X}(\mathscr{F},\fr{n}_{X/Z})  \ar[r]^-{\Tr_u}\ar[dr]_{\Tr_{v\circ u}} & v_*\mathbf{R}\Hom_{\mathcal{O}_Y}(u_*\mathscr{F},\fr{n}_{Y/Z})[c] \ar[d]^{\Tr_v} \\ & \mathbf{R}\Hom_{\mathcal{O}_Z}((v\circ u)_*\mathscr{F},\mathcal{O}_Z)[c+d] } \]
commutes. 
\end{lemma}

\begin{remark} Note that pushforward along a regular closed immersion preserves perfect complexes, which can be seen, for example by considering the Koszul complex of a regular generating sequence of the corresponding ideal sheaf.
\end{remark}

\begin{proof} This is essentially a case of carefully combining Lemma \ref{lemma: perf} above with coherent duality for schemes treated in \cite{Har66}. First of all, we define a functor
\[ u^\flat:= u^{-1}\mathbf{R}\Hom_{\mathcal{O}_Y}(u_*\mathcal{O}_X,-) : {\bf D}^+_\mathrm{coh}(\cO_Y)\rightarrow {\bf D}^+_\mathrm{coh}(\cO_X), \]
that this does indeed land in ${\bf D}^+_\mathrm{coh}(\cO_X)$ can be checked locally on $Y$, whence it follows from Lemma \ref{lemma: perf} together with the corresponding result for schemes \cite[Chapter III, Proposition 6.1]{Har66}. Next, the canonical morphism
\[ \mathbf{L}u^*u_*\mathcal{O}_X \rightarrow \mathcal{O}_X \]
induces, for any $\mathscr{F}\in {\bf D}^+_\mathrm{coh}(\cO_X)$, a map
\[ \mathscr{F}= \mathbf{R}\Hom_{\mathcal{O}_X}(\mathcal{O}_X,\mathscr{F}) \rightarrow  \mathbf{R}\Hom_{\mathcal{O}_X}(\mathbf{L}u^*u_*\mathcal{O}_X,\mathscr{F}) = u^{-1}\mathbf{R}\Hom_{\mathcal{O}_Y}(u_*\mathcal{O}_X,u_*\mathscr{F}) = u^\flat u_*\mathscr{F}, \]
which we claim induces an isomorphism
\[ \mathbf{R}{\rm Hom}_{\mathcal{O}_Y}(u_*\mathscr{F},\scr{G})\to\mathbf{R}{\rm Hom}_{\mathcal{O}_X}(\mathscr{F},u^\flat\scr{G})  \]
for any $\scr{G}\in {\bf D}^+_{\rm coh}(\cO_Y)$. Since the map $\scr{F}\to u^\flat u_*\scr{F}$ is defined globally, the fact that it defines such an adjunction can be checked locally, when again it follows from Lemma \ref{lemma: perf} together with the analogous result for schemes \cite[Chapter III, Theorem 6.7]{Har66}. Now uniqueness of adjoints gives rise to a canonical isomorphism $(v\circ u)^\flat\cong u^\flat\circ v^\flat$ whenever $X\overset{u}{\rightarrow}Y \overset{v}{\rightarrow} Z$ is a pair of regular closed immersions between adic spaces, and this isomorphism can, locally, be identified with that from \cite[Chapter III, Proposition 6.2]{Har66}.

The first claim therefore reduces to constructing a natural isomorphism
\[ \chi_u:u^\flat\mathcal{O}_Y \cong \mathfrak{n}_{X/Y}[-c]. \]
in ${\bf D}^+_\mathrm{coh}(\cO_X)$. Given this, the second claim then boils down to showing that if $X\overset{u}{\rightarrow}Y \overset{v}{\rightarrow} Z$ is a pair of regular closed immersions between adic spaces, then the diagram
 \[ \xymatrix{ (v\circ u)^\flat\mathcal{O}_Z  \ar[d]_{\mathrm{canonical}}\ar[rr]^-{\chi_{v\circ u}} & & \mathfrak{n}_{X/Z}[-c-d] \ar[d]^{\mathrm{canonical}} \\
u^\flat v^\flat\mathcal{O}_Z \ar[r]^-{\chi_v} & u^\flat\mathfrak{n}_{Y/Z}[-d] \ar[r]^-{\chi_u} & u^*\mathfrak{n}_{Y/Z} \otimes_{\mathcal{O}_X} \mathfrak{n}_{X/Y}[-c-d]    } \]
commutes. Since the first claim in particular implies that the relative dualising complex $u^\flat\mathcal{O}_Y$ is concentrated in a single degree, they may be jointly checked locally on $Z$. Thus we may assume, in the first case, that $X$ is cut out by a global regular sequence in $Y$, and in the second case, that moreover $Y$ is also cut out by a global regular sequence in $Z$. Under these assumptions, the claims are are a consequence of the analogous results for schemes, in particular the calculation of $u^\flat$ for a regular closed immersion in \cite[Corollary 7.3]{Har66}.
\end{proof}

\subsection{Closed subspaces of open polydiscs} \label{sec: trace closed}

We will apply the results of \S\ref{sec: closed dual} to a closed immersion $u:X\rightarrow \D^N_Y(0;1^-)$ of adic spaces, over a finite dimensional adic space $Y$, such that the composite $f:=\pi\circ u$ 
\[ X\overset{u}{\hookrightarrow} \D^N_Y(0;1^-) \overset{\pi}{\rightarrow} Y \]
of $u$ with the natural projection $\pi$ is smooth of relative dimension $d$. Since $\omega_{X/Y}\cong \fr{n}_{X/\D^N_Y(0;1^-)}\otimes_{\mathcal{O}_X} u^*\omega_{\D^N_Y(0;1^-)/Y}$, by taking $\mathcal{F}=\mathcal{O}_X$, tensoring both sides with $\omega_{\D^N_Y(0;1^-)/Y}$, and using the projection formula, we obtain an isomorphism
\[ \Tr_u: u_*\omega_{X/Y} \isomto \mathbf{R}\Hom_{\mathcal{O}_{\D^N_Y(0;1^-)}}(u_*\mathcal{O}_X,\omega_{\D^N_Y(0;1^-)/Y})[N-d]. \]
Hence applying $\mathbf{R}^d\pi_!$ gives an isomorphism
\[ 
\mathbf{R}^df_!\omega_{X/Y} \isomto \mathbf{R}^N\pi_!\mathbf{R}\Hom_{\mathcal{O}_{\D^N_Y(0;1^-)}}(u_*\mathcal{O}_X,\omega_{\D^N_Y(0;1^-)/Y}). \]
Restricting along $\mathcal{O}_{\D^N_Y(0;1^-)}\rightarrow u_*\mathcal{O}_X$ gives a map
\[
\mathbf{R}^df_!\omega_{X/Y} \rightarrow \mathbf{R}^N\pi_!\omega_{\D^N_Y(0;1^-)/Y}, \]
and finally composing with $\Tr_{z_1,\ldots,z_N}$ for a choice of co-ordinates on $\D^N_Y(0;1^-)$ gives a trace map
\[ \Tr_{X/Y}: \mathbf{R}^df_!\omega_{X/Y} \rightarrow \mathcal{O}_Y.\]
Via Theorem \ref{theo: coh dim pp smooth} we may view this as a map
\[\mathbf{R}f_!\omega_{X/Y}[d] \rightarrow \mathcal{O}_Y.\]

\begin{proposition} \label{prop: tr1} Suppose that $Y$ is a finite dimensional adic space, and $f:X\rightarrow Y$ is a smooth morphism of relative dimension $d$, factoring through a closed immersion into an open unit polydisc over $Y$. 
\begin{enumerate}
\item The induced map $\Tr_{X/Y}: \mathbf{R}f_!\omega_{X/Y}[d]\rightarrow \mathcal{O}_Y$ does not depend on the choice of embedding $u:X\hookrightarrow \D^N_Y(0;1^-)$ over $Y$.
\item Suppose that $g\colon Y \rightarrow Z$ is a smooth morphism of relative dimension $e$, factoring through a closed embedding into some relative open disc $\D^{M}_Z(0;1^-)$. Then, via the identification $\omega_{X/Z} = \omega_{X/Y}\otimes f^*\omega_{Y/Z}$, the diagram
\[ \xymatrix{ \mathbf{R}(g\circ f)_!\omega_{X/Z}[d+e] \ar[rr]^-{\mathbf{R}g_!(\mathrm{Tr}_{X/Y})}\ar[drr]_{\Tr_{X/Z}} & & \mathbf{R}g_!\omega_{Y/Z}[e] \ar[d]^{\Tr_{Y/Z}} \\ & & \mathcal{O}_Y   } \]
commutes.
\item The trace map vanishes on the image of 
\[ \mathbf{R}^df_!\Omega^{d-1}_{X/Y} \rightarrow \mathbf{R}^df_!\omega_{X/Y}, \]
and hence descends to a map
\[ \mathrm{Tr}_{X/Y}: \mathbf{R}f_!\Omega^\bullet_{X/Y}[2d]\rightarrow \mathcal{O}_Y. \] 
\end{enumerate}	
\end{proposition}

\begin{proof}  For ease of notation, we will drop $(0;1^-)$ from the notation for open polydiscs during the proof. Given Lemma \ref{lemma: basic change 2}, part (1) is proved verbatim as in the case of closed subspace of analytic affine space over a height one affinoid field treated in \cite[Theorem 3.5]{vdP92} (which is in turn based on \cite[\S5]{Chi90}). For part (2), we can choose a commutative diagram
\begin{equation}
\label{eqn: triple triangle}
 \xymatrix{ X \ar[r]^-u \ar[dr]_-f & \D^{N}_Y \ar[r]^-v \ar[d]^-\pi & \D^{N+M}_Z \ar[d]^\pi \\ & Y \ar[r]^-v \ar[dr]_-g & \D^M_Z \ar[d]^-\rho \\ & & Z } 
\end{equation}
with all horiozntal arrows closed immersions, all vertical arrows the natural projections, and the upper right hand square Cartesian. Writing things out painfully explicitly, we need to show that the diagram below is commutative.
\[  \hspace*{-6em} \xymatrix@C=-4em{    \mathbf{R}(gf)_!\omega_{X/Z}[d+e] \ar@{=}[r]\ar@{=}[d] & \mathbf{R}g_!(\mathbf{R}f_!\omega_{X/Y} \otimes \omega_{Y/Z})[d+e] \ar@{=}[r] & \mathbf{R}g_!(\mathbf{R}\pi_!u_*\omega_{X/Y} \otimes \omega_{Y/Z})[d+e] \ar[d]^{\Tr_u} \\
 \mathbf{R}(\rho\pi)_!(vu)_*\omega_{X/Z}[d+e]  \ar[dd]_{\Tr_{vu}} \ar[dr]^{\Tr_u} & &  \mathbf{R}g_!(\mathbf{R}\pi_!\mathbf{R}\underline{\mathrm{Hom}}(u_*\mathcal{O}_X,\omega_{\D^N_Y/Y})\otimes \omega_{Y/Z})[N+e] \ar[dd]^{\mathrm{res}} \\
  & \mathbf{R}(\rho\pi)_!v_*(\mathbf{R}\underline{\mathrm{Hom}}(u_*\cO_X,\omega_{\D^N_Y/Y})\otimes \pi^* \omega_{Y/Z})\ar[dl]_{\Tr_v}[N+e] \ar@{=}[ur] &   \\ 
   \mathbf{R}(\rho\pi)_!\mathbf{R}\underline{\mathrm{Hom}}((vu)_*\mathcal{O}_X,\omega_{\D^{N+M}_Z/Z})[N+M] \ar[d]_{\mathrm{res}} & (\star) &  \mathbf{R}g_!(\mathbf{R}\pi_!\omega_{\D^N_Y/Y}\otimes \omega_{Y/Z})[N+e] \ar[d]^{\Tr_{\D^N_Y/Y}} \\ 
 \mathbf{R}(\rho\pi)_!\omega_{\D^{N+M}_Z/Z}[N+M] \ar[dd]_{\Tr_{\D^{N+M}_Z/Z}} \ar[dr]^{\Tr_{\D^{N+M}_Z/\D^M_Z}}  &  & \mathbf{R}g_!\omega_{Y/Z}[e] =\mathbf{R}\rho_!v_*\omega_{Y/Z}[e]\ar[d]^{\Tr_v} \\
 & \mathbf{R}\rho_!\omega_{\D^M_Z/Z}[M] \ar[dl]_{\Tr_{\D^M_Z/Z}} & \mathbf{R}\rho_!\mathbf{R}\underline{\mathrm{Hom}}(v_*\mathcal{O}_Y,\omega_{\D^M_Z/Z})[M] \ar[l]_-{\mathrm{res}}  \\ 
 \mathcal{O}_Z & &  }    \]
Commutativity of the two left hand triangles follows from Lemma \ref{lemma: closed dual} and Proposition \ref{prop: tr basic} respectively, and commutativity of the upper hexagon is straightforward. It therefore suffices to prove commutativity of the central octagon $(\star)$ which essentially expresses the fact that the maps $\Tr_v$ and $\Tr_{\D^N}$ arising from from the Cartesian upper right hand square of (\ref{eqn: triple triangle}) commute. 

To prove this commutativity, we may remove $\mathbf{R}\rho_!$ from every term appearing in this octagon, as well as tensor everything in sight by (an appropriate pullback of) $\omega_{\D^M_Z/Z}^{\otimes -1}$, and shift by $-M$. We will also simplify notation slightly, and replace $\D^M_Z$ by $Z$, thus the upper right hand square becomes
\[ \xymatrix{ \D^{N}_Y \ar[d]_\pi \ar[r]^v & \D^{N}_Z \ar[d]^\pi \\ Y \ar[r]^v & Z. } \] 
Finally, we can replace $u_*\mathcal{O}_X$ by $\cO_{\D^N_Y}$, which has the effect of removing the right hand map labelled `$\mathrm{res}$', in the diagram, thus turning the octagon into a heptagon. The diagram that we need to show commutes is therefore the one below, where $c$ is the codimension of $Y$ in $Z$.
\[ \xymatrix@C=0em{ & \mathbf{R}\pi_!v_*\left( \omega_{\D^N_Y/Y} \otimes \pi^*\mathfrak{n}_{Y/Z} \right)[N-c] \ar@{=}[dr]\ar[dl]_{{\Tr}_v} \\ \mathbf{R}\pi_!\mathbf{R}\Hom(v_*\mathcal{O}_{\D^{N}_Y},\omega_{\D^{N}_Z/Z})[N]\ar[d]_{\mathrm{res}} & &  \ar[d]^{\Tr_{\D^N_Y/Y}} v_*\mathbf{R}\pi_!\omega_{\D^N_Y/Y}\otimes \mathfrak{n}_{Y/Z} [N-c] \\ \mathbf{R}\pi_! \omega_{\D^{N}_Z/Z}[N]  \ar[dr]^{{\Tr}_{\D^N_Z/Z}} & & v_*\mathfrak{n}_{Y/Z}[-c] \ar[d]^{{\Tr}_v} \\ & \cO_Z  & \mathbf{R}\Hom(v_*\cO_Y,\cO_Z) \ar[l]_{\mathrm{res}} }  \]
Now, by localising on $Z$, we can assume that $Y$ is defined in $Z$ by a regular sequence $f_1,\ldots,f_c\in \Gamma(Z,\mathcal{O}_Z)$. Choosing co-ordinates $z_1,\ldots,z_N$ on $\D^N$, we may use $z_1,\ldots,z_N$ and $f_1,\ldots,f_c$ to trivialise the canonical and normal sheaves appearing in the above diagram. We then consider the Koszul resolution
\[ \mathcal{K}^\bullet_{Y/Z}:= \left[ \mathcal{O}_Z \longrightarrow \ldots \longrightarrow \mathcal{O}_Z^{\oplus \binom{c}{2}}\longrightarrow \mathcal{O}_Z^{\oplus c} \overset{(f_1,\ldots,f_c)}{\longrightarrow}\mathcal{O}_Z \right] \]
of $v_*\mathcal{O}_Y$, viewed as being concentrated in the interval $[-c,0]$. Thus the natural map $\cO_Z\rightarrow v_*\cO_Y$ corresponds to the inclusion
\[ \iota_0:\cO_Z \rightarrow \mathcal{K}^\bullet_{Y/Z} \]
of the term in degree $0$, and the trace morphism $\Tr_v$ corresponds to the canonical isomorphism of complexes
\[ \mathrm{can}\colon \mathcal{K}^\bullet_{Y/Z}\overset{\cong}{\longrightarrow}  \Hom( \mathcal{K}^\bullet_{Y/Z},\cO_Z)[c]   \]
(see \cite[\S7]{Har66}). Of course, $\pi^*\mathcal{K}^\bullet_{Y/Z}$ is a resolution of $v_*\mathcal{O}_{\D^N_Y}$ as a $\cO_{\D^N_Z}$-module, with similar descriptions of $\Tr_v$ and $\cO_{\D^N_Z}\rightarrow v_*\cO_{\D^N_Y}$. The diagram that we are required to show the commutativity of then becomes the following.
\[ \xymatrix@C=0em{ & \mathbf{R}\pi_!\pi^*\mathcal{K}_{Y/Z}^\bullet [N-c] \ar@{=}[dr]\ar[dl]_{\mathrm{can}} \\ \mathbf{R}\pi_!\pi^*\Hom(\mathcal{K}_{Y/Z}^\bullet,\cO_Z)[N]\ar[d]_{\iota_0^*} & &  \ar[d]^{\Tr_{z_1,\ldots,z_N}} \mathbf{R}\pi_!\cO_{\D^N_Y}\otimes \mathcal{K}_{Y/Z}^\bullet[N-c] \\ \mathbf{R}\pi_! \cO_{\D^{N}_Z}[N]  \ar[dr]^{{\Tr}_{z_1,\ldots,z_N}} & & \mathcal{K}_{Y/Z}[-c] \ar[d]^{\mathrm{can}} \\ & \cO_Z  & \Hom(\mathcal{K}_{Y/Z}^\bullet,\cO_Z) \ar[l]_{\iota_0^*} }  \]
The claim now following from the explicit description of $\mathbf{R}\pi_!\cO_{\D^N}$ and the construction of the trace map $\Tr_{z_1,\ldots,z_N}$ in \S\ref{sec: relative open unit disc} above.

For part (3), the question is local on $Y$, and on $X$ by Corollary \ref{cor: MV ss}. Hence we may assume that there exists a closed immersion $X\hookrightarrow \D^N_Y(0;1^-)$ over $Y$. We may also assume that $Y$ is Tate affinoid, with quasi-uniformiser  $\varpi\in\Gamma(Y,\cO_Y)$.

Since $X$ is smooth over $Y$, the module of differentials $\Omega^1_{X/Y}$ is locally free. Since $Y$ is affinoid, it follows from Proposition \ref{prop: A and B} that we may choose, for any $x\in X$, functions $t_1,\ldots,t_d\in \Gamma(X,\mathcal{O}_X)$ such that ${\rm d}t_1,\ldots,{\rm d}t_d$ form a basis of $\Omega^1_{X/Y,x} \otimes_{\mathcal{O}_{X,x}}k(x)$, and thus (by Nakayama's lemma) a basis of $\Omega^1_{X/Y,x}$. The locus of points $x'$ where ${\rm d}t_1,\ldots,{\rm d}t_d$ are not a basis of $\Omega^1_{X/Y,x'}$ is then a closed analytic subspace of $X$.

We claim that the complement of any such subspace locally admits a closed immersion into an open unit polydisc over $Y$. To see this, let us set $X_n:=X\cap \D^N_Y(0;\norm{\varpi}^{\frac{1}{n}})$ for each $n\geq 0$. Then, for any $f\in \Gamma(X_n,\cO_X)$, the Zariski open subset $D(f):=\{x\in X_n\mid f(x)\neq 0\}$ of $X_n$ admits a closed immersion into $\A^{2,\an}_{X_n}$ via
\[ \left(f,\frac{1}{f}\right):D(f)\rightarrow \A^{2,\an}_{X_n}.\]
Hence, for any $m\in \Z_{\leq 0}$, $D(f)\cap \D^N_Y(0;\norm{\varpi}^{\frac{1}{n}-})\cap \D^N_{X_n}(0;\norm{\varpi}^{m-})$ admits a closed immersion into a unit polydisc over $Y$, which proves the claim.

We may therefore reduce to the case where ${\rm d}t_1,\ldots,{\rm d}t_d$ are a basis of $\Omega^1_{X/Y}$ on the whole of $X$. It then follows from  \cite[Proposition 1.6.9 iii)]{Hub96} that the morphism $\bm{t}:=(t_1,\ldots,t_d): X\rightarrow \A^{d,\an}_{Y}$ is \'etale. Further localising on $X$, we may assume that the image of this morphism lands inside $\D^d_Y(0;1^-)$. But now, by applying part (2) to the composition 
\[ X \overset{\bm{t}}{\rightarrow} \D^d_Y(0;1^-)\overset{\pi}{\rightarrow} Y, \]
we deduce that the diagram
\[ \xymatrix{ \mathbf{R}f_!\omega_{X/Y}[d] \ar[rrr]^-{\mathbf{R}\pi_!(\mathrm{Tr}_{X/\D^d_Y(0;1^-)})}\ar[drrr]_{\Tr_{X/Y}} & & & \mathbf{R}\pi_!\omega_{\D^d_Y(0;1^-)/Y}[d] \ar[d]^{\Tr_{\D^d_Y(0;1^-)/Y}} \\ & & & \mathcal{O}_Y   } \]
commutes. Since the diagram
\[ \xymatrix{ \mathbf{R}f_!\Omega^{d-1}_{X/Y} \ar[rrr]^-{\mathbf{R}\pi_!(\mathrm{Tr}_{X/\D^d_Y(0;1^-)})} \ar[d]_{\mathbf{R}f_!({\rm d})} & & &  \mathbf{R}\pi_!\Omega^{d-1}_{\D^d_Y(0;1^-)/Y} \ar[d]^{\mathbf{R}\pi_!({\rm d})} \\ \mathbf{R}f_!\omega_{X/Y} \ar[rrr]^-{\mathbf{R}\pi_!(\mathrm{Tr}_{X/\D^d_Y(0;1^-)})} & & & \mathbf{R}\pi_!\omega_{\D^d_Y(0;1^-)/Y} } \]
also commutes, the claim therefore reduces to the case $X=\D^{d}_{Y}(0;1^-)$ that we have already handled.
\end{proof}

\begin{corollary} If $X=\D^d_Y(0;1^-)$ then the trace map
\[ \Tr_{z_1,\ldots,z_d}: \mathbf{R}^df_!\omega_{X/Y}\rightarrow \mathcal{O}_Y \]
defined above doesn't depend on the choice of co-ordinates $z_1,\ldots,z_d$.
\end{corollary}

\subsection{Smooth morphisms of adic spaces} \label{sec: tm as}

As a penultimate case, we construct the trace morphism when $Y$ is a finite dimensional adic space, and $f:X\rightarrow Y$ is a smooth morphism of relative dimension $d$, which is moreover partially proper in the sense of Kiehl. Then, locally on $Y$, there exists a cover of $X$ by opens $U_i$ admitting closed embeddings $U_i\hookrightarrow \D^{N_{i}}_Y(0;1^-)$ over $Y$. Moreover, each $U_i\cap U_j$ admits a closed embedding into $\D^{N_{i}+N_j}_Y(0;1^-)$ over $Y$. Using the spectral sequence from Corollary \ref{cor: MV ss}, together with Proposition \ref{prop: tr1}, the trace maps
\[ \Tr_{U_i/Y}: \mathbf{R}^{d}f_!\omega_{U_i/Y}\rightarrow \mathcal{O}_Y \]
factor uniquely through a map
\[ \Tr_{X/Y} : \mathbf{R}^{d}f_!\omega_{X/Y}\rightarrow \mathcal{O}_Y \]
which can be checked not to depend on the choice of the $U_i$ using Proposition \ref{prop: tr1}. Since $\Tr_{X/Y}$ does not depend on the choice of open cover of $X$, it glues over an open cover of $Y$, and by Theorem \ref{theo: coh dim pp smooth} this can be viewed as a map
\[ \mathrm{Tr}_{X/Y}: \mathbf{R}f_!\omega_{X/Y}[d]\rightarrow \cO_Y \]
in ${\bf D}(\mathcal{O}_Y)$. Moreover, $\mathrm{Tr}_{X/Y}$ vanishes on the image of
\[ \mathbf{R}^df_!\Omega^{d-1}_{X/Y} \rightarrow \mathbf{R}^df_!\omega_{X/Y}, \]
since the same is true locally on $X$ and $Y$, and thus induces a map
\[ \mathrm{Tr}_{X/Y}: \mathbf{R}f_!\Omega^\bullet_{X/Y}[2d]\rightarrow \mathcal{O}_Y \]
in ${\bf D}(\mathcal{O}_Y)$.

\begin{remark} When $X=\A^{d,\an}_Y$ there are two candidates for a trace map: one constructed immediately above, and the other alluded to in Remark \ref{rem: A^d trace}. The two are easily seen to coincide.
\end{remark}

\subsection{The general case}

Finally, we consider again a smooth morphism $f:X\rightarrow Y$ of relative dimension $d$, partially proper in the sene of Kiehl, but now with the base $Y$ allowed to be any overconvergent, finite dimensional germ. Then arguing exactly as in the proof of Corollary \ref{cor: coh dim omega germs}, we see that, locally on $X$, we can extend $f$ to a diagram of pairs
\[ \xymatrix{ (X,\bm{X}) \ar[r]^-{u} \ar[dr]_f & (\D^{N}_Y(0;1^-),\D^{N}_{\bm{Y}}(0;1^-)) \ar[d]^{\pi} \\ & (Y,\bm{Y})   } \]
such that $f:\bm{X}\rightarrow \bm{Y}$ is smooth, $u:\bm{X}\rightarrow \D^{N}_{\bm{Y}}(0;1^-)$ is a closed immersion over $\bm{Y}$, and $X=\pi^{-1}(Y)\cap \bm{X}$. Using Corollary \ref{cor: coh dim omega germs}, together with Lemma \ref{lemma: base change overconvergent closed}, we can therefore carry through all the arguments of \S\ref{sec: trace closed} and \S\ref{sec: tm as} above to construct a trace morphism
\[ \Tr_{X/Y}:\mathbf{R}^df_!\omega_{X/Y}\rightarrow \mathcal{O}_Y, \]
which again can be viewed as a map
\[ \Tr_{X/Y}:\mathbf{R}f_!\omega_{X/Y}[d]\rightarrow \mathcal{O}_Y.\]

\begin{proposition} \label{prop: trace main} Let $Y$ be an overconvergent, finite dimensional germ, and $f:X\rightarrow Y$ a smooth morphism of relative dimension $d$, which is partially proper in the sense of Kiehl.
\begin{enumerate}
\item $\Tr_{X/Y}$ vanishes on the image of $\mathbf{R}^df_!\Omega^{d-1}_{X/Y}$, and descends to a map
\[ \Tr_{X/Y}: \mathbf{R}f_!\Omega^\bullet_{X/Y}[2d] \rightarrow \mathcal{O}_Y. \]
\item If $g:Y\rightarrow Z$ is smooth morphism of relative dimension $e$, partially proper in the sense of Kiehl, with $Z$ overconvergent, then the diagram
\[ \xymatrix{  \mathbf{R}(g\circ f)_!\Omega^\bullet_{X/Z}[2d+2e] \ar[rrr]^-{\mathbf{R}g_!(\mathrm{Tr}_{X/Y})} \ar[drrr]_-{\mathrm{Tr}_{X/Z}} & && \mathbf{R} g_!\Omega^\bullet_{Y/Z}[2e] \ar[d]^-{\mathrm{Tr}_{Y/Z}} \\ & && \mathcal{O}_Z }
 \]
commutes.
\end{enumerate}
\end{proposition}

\subsection{Duality morphism}

We can now construct the duality morphism. Let $f:X\rightarrow Y$ be a partially proper morphism of germs. If $\mathscr{I},\mathscr{J}$ are $\mathcal{O}_X$-modules then the natural map
\[ {\rm H}^0(X,\mathscr{I}) \times {\rm H}^0(X,\mathscr{J}) \rightarrow {\rm H}^0(X,\mathscr{I} \otimes \mathscr{J} ) \]
induces
\[ {\rm H}^0(X,\mathscr{I}) \times {\rm H}^0_c(X/Y,\mathscr{J}) \rightarrow {\rm H}^0_c(X/Y,\mathscr{I} \otimes \mathscr{J}), \]
and sheafifying this gives a pairing
\[ f_*\mathscr{I} \times f_!\mathscr{J} \rightarrow f_!(\mathscr{I}\otimes \mathscr{J}). \]
By taking resolutions, we deduce that if $\mathscr{E}$ and $\mathscr{F}$ are bounded complexes of $\mathcal{O}_X$-modules, and both $X$ and $Y$ are finite dimensional, then there is a natural pairing
\[ \mathbf{R}f_*\mathscr{E} \times \mathbf{R}f_!\mathscr{F} \rightarrow \mathbf{R}f_!\left(\mathscr{E}\otimes^{\mathbf{L}} \mathscr{F}\right) \]
in ${\bf D}^-(\mathcal{O}_Y)$. In particular, if $\mathscr{G}$ is a third bounded complex, then any pairing
\[ \mathscr{E}\times \mathscr{F} \rightarrow \mathscr{G} \]
induces a corresponding pairing
\[ \mathbf{R}f_*\mathscr{E} \times \mathbf{R}f_!\mathscr{F} \rightarrow \mathbf{R}f_!\mathscr{G} \]
in cohomology. If we now assume, moreover, that:
\begin{itemize}
\item $f$ is smooth of relative dimension $d$, and partially proper in the sense of Kiehl;
\item $Y$ is overconvergent;
\item $\mathscr{E}$ is a perfect complex on $X$,
\end{itemize}
then, setting $\mathcal{E}^\vee:=\mathbf{R}\Hom(\mathcal{E},\mathcal{O}_X)$, we have a natural evaluation pairing
\[ \mathscr{E} \times \mathscr{E}^\vee \otimes \omega_{X/Y} \rightarrow \omega_{X/Y}. \]
Together with the trace map $\Tr_{X/Y}$ this induces a pairing
\[ \mathbf{R}f_*\mathscr{E} \times \mathbf{R}f_!\left( \mathscr{E}^\vee \otimes \omega_{X/Y} \right) \rightarrow \mathcal{O}_Y[-d]. \]
There is of course a similar pairing
\[ \mathbf{R}f_!\mathscr{E} \times \mathbf{R}f_*\left( \mathscr{E}^\vee \otimes \omega_{X/Y} \right) \rightarrow \mathcal{O}_Y[-d]. \]

\section{A counter-example} \label{sec: counter examples}

In this section we give a counter-example showing that the formalism of $\mathbf{R}f_!$ cannot be extended beyond the partially proper case in any reasonable way. Our example also shows that the analogue of Lemma \ref{lemma: base change overconvergent closed} fails in general if $Z$ is replaced by a non-maximal point of $Y$. The example is based upon a suggestion of B. Le Stum.

\begin{theorem} There does not exist a way to define, for any morphism $f:X\rightarrow Y$ which is separated, locally of $^+$weakly finite type, and taut, a functor
\[ \mathbf{R}f_!:{\bf D}^+(X)\rightarrow {\bf D}^+(Y), \]
in such a way that:
\begin{enumerate}
\item $\mathbf{R}f_!$ agrees with Definition \ref{defn: Rf! pp} given above whenever $f$ is partially proper;
\item $\mathbf{R}f_!=f_!$ is the extension by zero functor whenever $f$ is an open immersion;
\item $\mathbf{R}(g\circ f)_!\cong \mathbf{R}g_!\circ \mathbf{R}f_!$ whenever $f$ and $g$ are composable morphisms.
\end{enumerate}
\end{theorem}

\begin{proof}
Let $\kappa=(k,k^\circ)$ be a height one affinoid field, $X=\D^1_\kappa(0;1)=\spa{k\tate{z},k^\circ\tate{z}}$ the (ordinary) closed unit disc over $\kappa$. Let $j_U:U \rightarrow X$ denote the inclusion of the open subspace defined by $\{x\in X \mid v_x(z-1)=1\}$, $f:X\rightarrow X$ the (finite, thus proper) morphism defined by $z\mapsto z^2$, and $V:=f^{-1}(U)\subset X$. Thus $V$ is the intersection of $U$ with the open subspace defined by $\{x \in X \mid v_x(z+1)=1\}$, we let $j_V\from V\to X$ denote the inclusion. We therefore have a Cartesian square
\[ \xymatrix{  V\ar[r]^{j_V} \ar[d]_f & X \ar[d]^f \\ U \ar[r]^{j_U} & X. } \]
To prove the theorem, it suffices to show that $\mathbf{R}f_!\circ j_{V!}\ncong j_{U!}\circ  \mathbf{R}f_!$. 

To see this, we take $\mathscr{F}$ to be the constant sheaf $\underline{\Z}$ on $V$, and compute the stalks of both sides at the Type V apex point $\xi_1$ of the open disc $\{x\in X \mid v_{[x]}(z-1)<1\}$. Clearly we have that $\left(j_{U!}\mathbf{R}f_!\underline{\Z}\right)_{\xi_1}=0$, and we shall show that $\left(\mathbf{R}f_!j_{V!}\underline{\Z}\right)_{\xi_1}=\left(\mathbf{R}f_*j_{V!}\underline{\Z}\right)_{\xi_1}\neq 0$.

Indeed, we may base change to the set $G(\xi_1)$ of generalisations of $\xi_1$, which is a two point space $\{\xi_1,\xi\}$ consisting of $\xi_1$ together with the Gauss point $\xi$. The fibre $X_{(\xi_1)}=f^{-1}(G(\xi_1))$ is a three point space $\{\xi,\xi_1,\xi_{-1}\}$ consisting of $\xi,\xi_1$ and the apex point $\xi_{-1}$ of the open disc $\{x\in X \mid v_{[x]}(z+1)<1\}$, and the induced map $f:X_{(\xi_1)}\rightarrow G(\xi_1)$ sends $\xi_{\pm 1}$ to $\xi_1$ and $\xi$ to $\xi$. We then have
\[ \left(\mathbf{R}f_*j_{V!}\underline{\Z}\right)_{\xi_1} = \mathbf{R}\Gamma(X_{(\xi_1)},(j_{V!}\underline{\Z})|_{X_{(\xi_1)}}) \]
and we can identify the restriction of $j_{V!}\underline{\Z}$ to $X_{(\xi_1)}$ as the extension by zero of the constant sheaf $\underline{\Z}$ along the open immersion $\{ \xi \}\rightarrow X_{(\xi_1)}$. In particular, we have the exact sequence
\[ 0 \rightarrow \left(j_{V!}\underline{\Z}\right)|_{X_{(\xi_1)}} \rightarrow \underline{\Z}\rightarrow i_{1*}\underline{\Z} \oplus i_{-1*}\underline{\Z} \rightarrow 0 \]
where $i_{\pm1}$ is the inclusion of the closed point $\xi_{\pm1}$ inside $X_{(\xi_1)}$. Taking cohomology then gives the exact triangle
\[ \left(\mathbf{R}f_*j_{V!}\underline{\Z}\right)_{\xi_1}\rightarrow \Z \rightarrow \Z\oplus \Z \overset{+1}{\rightarrow} \]
where the second map is the diagonal morphism. Thus
\[ \left(\mathbf{R}f_*j_{V!}\underline{\Z}\right)_{\xi_1}\cong \Z[-1] \]
is non-zero as claimed.
\end{proof}

\providecommand{\bysame}{\leavevmode\hbox to3em{\hrulefill}\thinspace}
\providecommand{\MR}{\relax\ifhmode\unskip\space\fi MR }
\providecommand{\MRhref}[2]{%
  \href{http://www.ams.org/mathscinet-getitem?mr=#1}{#2}
}
\providecommand{\href}[2]{#2}

\end{document}